\documentclass{amsart}
\usepackage[utf8]{inputenc} % Required for inputting international characters
\usepackage[T1]{fontenc} % Output font encoding for international characters

\usepackage{amsmath, amsthm, amssymb, mathrsfs,amsrefs}
\usepackage{enumerate}
\usepackage{tikz-cd}
\usetikzlibrary{arrows}

\usepackage{hyperref}

\usetikzlibrary{matrix,arrows,decorations.pathmorphing}

 \let\temp\phi
\let\phi\varphi
\let\varphi\temp

\newcommand{\C}{\mathbb{C}}

\newcommand{\N}{\mathbb{N}}
\newcommand{\R}{\mathbb{R}}

\DeclareMathOperator{\Span}{span}

\newcommand{\calM}{\mathcal{M}}

\newcommand{\calP}{\mathcal{P}}

\newcommand{\calS}{\mathcal{S}}
\newcommand{\calT}{\mathcal{T}}

\DeclareMathOperator{\conv}{conv}

\newcommand{\Cmin}{C^\ast_{\text{min}}}

\DeclareMathOperator{\Real}{Re}
\DeclareMathOperator{\Imag}{Im}

\newcommand{\CB}{\text{CB}}

\theoremstyle{plain}
\newtheorem{lemma}{Lemma}[section]
\newtheorem{theorem}[lemma]{Theorem}

\newtheorem{proposition}[lemma]{Proposition}
\newtheorem{corollary}[lemma]{Corollary}

\theoremstyle{definition}
\newtheorem{question}[lemma]{Question}
\newtheorem{example}[lemma]{Example}
\newtheorem{remark}[lemma]{Remark}
\newtheorem{definition}[lemma]{Definition}

\usepackage{MnSymbol} %needed for \llangle and \rrangle
\usepackage{todonotes} %adds "todo notes"

\usepackage{hyperref}
\definecolor{darkblue}{rgb}{0.0,0.0,0.3}
\hypersetup{colorlinks,breaklinks,linkcolor=red,urlcolor=red,anchorcolor=red,citecolor=red}

\newtheorem{introtheoreminner}{Theorem}
\newenvironment{introtheorem}[1]{%
  \renewcommand\theintrotheoreminner{#1}%
  \introtheoreminner
}{\endintrotheoreminner}

\newcommand{\ncconv}{\mathrm{ncconv}}
\newcommand{\cncconv}{\overline{\mathrm{ncconv}}}
\newcommand{\real}{\text{Re }}
\newcommand{\im}{\text{Im }}

\newcommand{\sa}{\text{sa}}

\title[Extensions of nc convex sets and operator system duality]{An extension property for noncommutative convex sets and duality for operator systems}

\author[A. Humeniuk]{Adam Humeniuk}
\address{Department of Mathematics \& Computing\\ Mount Royal University\\
Calgary, Alberta, Canada}
\email{ahumeniuk@mtroyal.ca}

\author[M. Kennedy]{Matthew Kennedy}
\address{Department of Pure Mathematics\\ University of Waterloo\\
Waterloo, Ontario, Canada}
\email{matt.kennedy@uwaterloo.ca}

\author[N. Manor]{Nicholas Manor}
\address{Department of Pure Mathematics\\ University of Waterloo\\
Waterloo, Ontario, Canada}
\email{nmanor@uwaterloo.ca}

\begin{document}
\begin{abstract}
We characterize inclusions of compact noncommutative convex sets with the property that every continuous affine function on the smaller set can be extended to a continuous affine function on the larger set with a uniform bound. As an application of this result, we obtain a simple geometric characterization of (possibly nonunital) operator systems that are dualizable, meaning that their dual can be equipped with an operator system structure. We further establish some permanence properties of dualizability, and provide a large new class of dualizable operator systems. These results are new even when specialized to ordinary compact convex sets.

\end{abstract}


\thanks{Second author supported by NSERC Grant Number 418585.}

\maketitle

\tableofcontents

\section{Introduction}

%%%Dualizability
A unital operator system $S$ is a $\ast$-closed unital subspace of the bounded operators $B(H)$ on a Hilbert space $H$. In this paper, we assume that all operator spaces and operator systems are norm-complete. Choi and Effros \cite{choi1977injectivity} gave an abstract characterization of unital operator systems as matrix ordered $\ast$-vector spaces which contain an archimedean matrix order unit. In light of this, it is natural to ask if the dual space $S^\ast$ is an abstract unital operator system. 

%Using this characterization, it is natural to ask if the dual space $S^\ast$ is itself a unital operator system.

The dual $S^\ast$ is at least a complete operator space, and inherits a $\ast$-operation and matrix ordering from $S$. One says that $S^\ast$ is a \emph{matrix ordered operator space}. However, $S^\ast$ typically fails to have an order unit in infinite dimensions.
%For instance, if $S=C(X)$ is a commutative C*-algebra, then the dual $C(X)^\ast$ is the space of Radon measures on the compact Hausdorff space $X$, which never has an order unit if $X$ is uncountable.
So, one requires a theory of \emph{nonunital} operator systems if $S^\ast$ is to be an operator system.

Werner \cite{werner2002subspaces} defined nonunital operator systems, which we hereafter refer to simply as ``operator systems'', as matrix ordered operator spaces which embed completely isometrically and completely order isomorphically into $B(H)$. Werner gave an abstract characterization of operator systems that extends the Choi-Effros axioms in the unital setting. One would hope that $S^\ast$ is such an operator system, but it turns out that this is too much to ask for.

%For instance, we have the standard duality $M_n^\ast \cong M_n$, but this duality is not completely isometric. In fact, an embedding $M_n^\ast\hookrightarrow B(H)$ cannot be completely isometric and completely order isomorphic at the same time. However, the isomorphism $M_n^\ast\cong M_n$ is a \emph{complete isomorphism}, inducing completely equivalent matrix norms.

It is natural to say that an operator system $S$ is \emph{dualizable} if the dual matrix ordered operator space $S^\ast$ embeds into $B(H)$ via a map which is both a complete order isomorphism and is \emph{completely bounded below}. That is, $S^\ast$ can be re-normed with completely equivalent matrix norms in a way that makes it an operator system.

It is easy to construct examples of non-dualizable operator systems. For instance, the operator system $S$ consisting of $2 \times 2$ matrices with diagonal entries equal to zero, i.e.
\[
S = \text{span}\{E_{12},E_{21}\} \subseteq M_2,
\]
contains no nonzero positive elements, implying in particular that the set of positive elements in $S^\ast$ contains a line, precluding the existence of an order isomorphism into $B(H)$.

Recently, Ng \cite{ng2022dual} obtained an intrinsic characterization of dualizability for operator systems. Specifically, he proved that an operator system $S$ is dualizable if and only if it has the {\em completely bounded positive decomposition property}, meaning that there is a constant $C>0$ such that for every $n \in \mathbb{N}$ and every self-adjoint element $x\in M_n(S)^\sa$, there are positives $y,z\in M_n(S)^+$ with $x=y-z$ and $\|y\|+\|z\|\le C\|x\|$. Ng observed that every unital operator system $S$ is dualizable with $C=2$. Further, using the continuous functional calculus, he showed that every C*-algebra is dualizable with $C=1$. So the dualizable operator systems do form a large subclass of operator systems.

\bigskip

It turns out that the property of dualizability for operator systems is closely related to an important extension property for noncommutative convex sets. These are the main objects of interest in the theory of noncommutative convexity, which was recently introduced by Davidson and the second author \cite{davidson2019noncommutative}.

A noncommutative (nc) convex set $K$ over an operator space $E$ is a graded set
\[
K=\coprod_{n} K_n \subseteq \coprod_{n} M_n(E)
\]
that is closed under direct sums and compression by scalar isometries. Equivalently, $K$ is closed under nc convex combinations, meaning that $\sum \alpha_i^* x_i \alpha_i \in K_n$ for every bounded family of points $x_i \in K_{n_i}$ and every family of scalar matrices $\alpha_i \in M_{n_i,n}$. If each $K_n$ is compact, then $K$ is said to be compact.

An important point is that the above union is taken over all $n \leq \kappa$ for some sufficiently large cardinal number $\kappa$, with the convention $M_n=B(H_n)$ for a Hilbert space $H_n$ of dimension $n$. In the separable setting, it typically suffices to take $\kappa = \aleph_0$. This is in contrast to Wittstock's definition of a matrix convex set, where the union is taken over all $n < \aleph_0$. While every matrix convex set uniquely determines a nc convex set, the ability to work with points at infinity is essential for a robust theory (see  \cite{davidson2019noncommutative} for more details).

The duality theorem from \cite{davidson2019noncommutative}, which is essentially equivalent to a result of Webster-Winkler \cite{webster1999krein} for compact matrix convex sets, asserts that the category of compact nc convex sets is dual to the category of unital operator systems. In particular, every unital operator system is isomorphic to the operator system $A(K)$ of continuous affine nc functions on $K$, where $K = \calS(S)$ is the nc state space of $S$, i.e. 
\[
K_n = \{\phi:S\to M_n\mid \phi \text{ is unital and completely positive}\}.
\]

We say that an inclusion $K \subseteq L$ of compact nc convex sets has the {\em bounded extension property} if there is a constant $C > 0$ such that every continuous affine nc function $a \in A(K)$ extends to a continuous affine nc function $b \in A(L)$ such that $\|b\| \leq C \|a\|$. In Section \ref{sec:extension}, we establish the following characterization of the bounded extension property in terms of the operator space structure of the restriction map $A(L) \to A(K)$, and in terms of the geometry of the inclusion $K \subseteq L$.

\begin{introtheorem}{A} \label{thm:A}
Let $K \subseteq L$ be an inclusion of compact nc convex sets. The following are equivalent:
\begin{enumerate}
\item The inclusion $K \subseteq L$ has the bounded extension property.
\item The restriction map $A(L) \to A(K)$ is an operator space quotient map.
\item There is a constant $C > 0$ such that $(K - K) \cap \operatorname{span}_{\mathbb{R}} L \subseteq C(L - L)$.
\end{enumerate}
\end{introtheorem}

The duality theorem from \cite{kennedy2023nonunital} extends the duality theorem for unital operator systems to (possibly nonunital) operator systems. It asserts that the category of pointed compact nc convex sets is dual to the category of operator systems, where a pointed compact convex set is a pair $(K,z)$ consisting of a compact nc convex set $K$ and a distinguished point $z \in K_1$ that behaves like zero in a certain precise sense. In particular, every operator system is isomorphic to the operator system $A(K,z)$ of continuous affine nc functions on $K$ that vanish at $z = 0$, where $K = \mathcal{QS}(S)$ is the nc quasistate space of $S$, i.e.
\[
K_n = \{\phi:S\to M_n\mid \phi \text{ is completely contractive and completely positive}\}.
\]
This result was utilized to provide additional insight on some recent results of Connes and van Suijlekom \cite{connes2021spectral}.

In Section \ref{sec:dualizability}, as an application of Theorem \ref{thm:A} and the duality between operator systems and pointed compact nc convex sets, we obtain a simple new geometric characterization of dualizability for operator systems.

\begin{introtheorem}{B} \label{thm:B}
Let $S$ be an operator system with nc quasistate space $K$. The following are equivalent: 
\begin{enumerate}
\item The operator system $S$ is dualizable.
\item There is a constant $C > 0$ such that $(K - \mathbb{R}_+ K) \cap \mathbb{R}_+ K \subseteq CK$.
\item The set $(K - \mathbb{R}_+ K) \cap \mathbb{R}_+ K$ is bounded.
\end{enumerate}
\end{introtheorem}

In Section \ref{sec:nc-simplices}, we prove the dualizability of a large new class of operator systems. Specifically, we consider operator systems with nc quasistate spaces that are nc simplices in the sense of \cite{kennedy2022noncommutative}. Although this class of operator systems contains all C*-algebras, it is much larger.

\begin{introtheorem}{C} \label{thm:C}
Let $S$ be an operator system with nc quasistate space $K$. If $K$ is an nc simplex and $0 \in K_1$ is an extreme point in $K$, then $S$ is dualizable.
\end{introtheorem}

In Section \ref{sec:commutative}, we discuss the existing corresponding classical duality theory for function systems, which are the commutative operator systems. Here, the dual of a function system is again equivalent to a function system if we have (ordinary) bounded positive decomposition. By completeness, this is equivalent to ordinary positive generation. However, we don't know if positive generation ensures completely bounded positive generation for function systems.

In Section \ref{sec:positive_generation}, we relate the completely bounded positive decomposition property of an operator system to the property of being positively generated. In contrast to the classical situation for function systems in Section \ref{sec:commutative}, positive generation at all matrix levels is not enough to imply the completely bounded positive decomposition property. We provide an example of a matrix ordered operator space that is positively generated but does not have the bounded positive decomposition property. However, we do show that positive generation of an operator system at the first level is enough to automatically imply positive generation at all matrix levels.

Finally, in Section \ref{sec:permanence}, we establish some permanence properties of dualizability, showing that quotients and pushouts of dualizable operator systems are again dualizable.

\section*{Acknowledgements}
The authors are grateful to David Blecher, Ken Davidson, Nico Spronk, Ivan Todorov and Vern Paulsen for their helpful comments and suggestions. The authors are also grateful to C.K. Ng for providing a preprint of \cite{ng2022dual}.

\section{Background}\label{sec:background}

\subsection{Nonunital operator systems}

All vector spaces in this paper are over $\C$, unless stated otherwise. If $V$ is a vector space and $n\in \N$, we let $M_n(V)$ be the vector space of $n\times n$-matrices with entries in $V$. We will typically identify $M_n(V)$ with $M_n\otimes V$, and write for instance
\[
1_2 \otimes x =
\begin{pmatrix}
	x & 0 \\
	0 & x
\end{pmatrix},
\]
where $x\in V$ and $1_2\in M_2$ is the identity matrix. We will also use the notation
\[
\calM(V):=\coprod_{n\ge 1} M_n(V)
\]
to denote the matrix universe over $V$.

If $V$ is any normed vector space and $r\ge 0$, we will frequently use $B_r(V)$ to denote the closed ball in $V$ with radius $r$ and center $0\in V$.

Following Ng \cite{ng2022dual}, we fix the following definitions. An \textbf{operator space} $E$ is a vector space equipped with a complete family of $L^\infty$-matrix norms, which we will denote either by $\|\cdot\|$, $\|\cdot\|_E$, or $\|\cdot\|_{M_n(E)}$ as appropriate.

\begin{definition}\label{def:mos}
A \textbf{semi-matrix ordered operator space} $(X,P)$ consists of an operator space $X$ equipped with a conjugate-linear completely isometric involution $x\mapsto x^\ast$, and a distinguished selfadjoint matrix convex cone $P=\coprod_{n\ge 1}P_n\subseteq \coprod_n M_n(X)^\sa$ such that each $P_n$ is norm-closed in $M_n(X)$. Usually we omit the symbol $P$ and write $M_n(X)^+:=P_n$. If in addition each cone $M_n(X)^+$ satisfies $M_n(X)^+\cap (-M_n(X)^+)=\{0\}$, then we say $X$ is a \textbf{matrix ordered operator space}. If $X$ is in addition a dual space $X=(X_\ast)^\ast$, we say $X$ is a \textbf{dual matrix ordered operator space} if the positive cones $M_n(X)^+$ are weak-$\ast$ closed.
\end{definition}

\begin{definition}\label{def:positively_generated}
A semi-matrix ordered operator space $X$ is \textbf{positively generated} if
\[
M_n(X)^\sa = M_n(X)^+-M_n(X)^+
\]
for all $n\ge 1$.
\end{definition}

\begin{example}\label{eg:dual_mos}
If $X$ is a positively generated matrix ordered operator space, then $X^\ast$ is naturally a dual matrix ordered operator space with the standard norm and order structure that identifies
\begin{align*}
M_n(X^\ast) & \cong \text{CB}(X,M_n) \quad\text{ isometrically, and}\\
M_n(X^\ast)^+ & \cong \text{CP}(X,M_n).
\end{align*}
\end{example}

\begin{definition}\label{def:mos_maps}
Let $X$ and $Y$ be matrix ordered operator spaces, and let $\phi:X\to Y$ be a linear map. For any $n\ge 1$, $\phi$ induces a linear map $\phi_n:M_n(X)\to M_n(Y)$. We say that $\phi$ is completely bounded, contractive, bounded below, isometric, positive, or a complete order isomorphism when each induced map $\phi_n$ satisfies the same property uniformly in $n$. If $\phi$ is completely bounded below and positive, we say $\phi$ is a \textbf{complete embedding}. If $\phi$ is completely isometric and positive, we say $\phi$ is a \textbf{completely isometric embedding}. If $\phi$ is also a linear isomorphism, we call $\phi$ a \textbf{complete isomorphism} or \textbf{completely isometric isomorphism} as appropriate.
\end{definition}

The class of all matrix ordered operator spaces forms a category, where one usually chooses the morphisms to be completely contractive and completely positive (ccp) maps, or completely bounded and completely positive (cbp) maps. In the interest of readability, we hereafter adopt the convention that ``\textbf{completely contractive and positive}" always means ``completely contractive and completely positive'', and similarly for ``completely bounded and positive''. That is, ``completely" modifies both the words ``contractive'' and ``positive''. Since we have no need to consider maps which are positive but not completely positive, there should be no risk of confusion.

\begin{example}\label{eg:operator_system_mos}
Let $S$ be a unital operator system, i.e. an $\ast$-matrix ordered space with archimedian matrix order unit $1_S$. Then $S$ is a matrix ordered operator space with norm
\[
\|x\|_{M_n(S)}=
\inf\left\{t\ge 0 \;\middle|\; \begin{pmatrix} t(1_n\otimes 1_s) & x \\
	x^\ast & t(1_n\otimes 1_s)
	\end{pmatrix}\ge 0 \text{ in }M_{2n}(S)^\sa
\right\}.
\]
This norm agrees with the induced norm from any unital complete order embedding $S\subseteq B(H)$. In particular, for any Hilbert space $H$, the space $B(H)$ is a unital operator system.
\end{example}

\begin{definition}\label{def:operator_system}
Let $S$ be a matrix ordered operator space. We say that $S$ is a \textbf{quasi-operator system} if there is a complete embedding $S\to B(H)$ for some Hilbert space $H$, and that $S$ is a \textbf{operator space} if there is a completely isometric embedding $S\to B(H)$. If $S$ is in addition a \emph{dual} matrix ordered operator space, then we say $S$ is a \text{dual (quasi-)operator system} if there is a weak-$\ast$ homeomorphic (complete embedding) completely isometric embedding into some $B(H)$.
\end{definition}

That is, a quasi-operator system $S$ is a matrix ordered operator space which is completely isomorphic to an operator system. Put another way, one can choose a completely equivalent system of norms on $S$, for which $S$ embeds completely isometrically and order isomorphically into $B(H)$.

%\begin{remark}
%If $S$ is a quasi-operator system, then Ng \cite[Corollary 3.14]{ng2022dual} showed that $S^\ast$ is a dual quasi-operator system if and only if it is a quasi-operator system.  That is, if $S^\ast$ embeds into $B(H)$ via a complete embedding, then there exists another Hilbert space $H'$ such that $S^\ast$ embeds into $B(H')$ via a \emph{weak-$\ast$ homeomorphic} complete embedding. So, when considering only quasi-operator systems $S$.
%\end{remark}

\subsection{Pointed noncommutative convex sets}

Suppose that $E=(E_\ast)^\ast$ is a dual operator space. Let
\[
\calM(E):=\coprod_{n\ge 1} M_n(E),
\]
where the union is taken over all cardinals $n\ge 1$ up to some fixed cardinal $\alpha$ at least as large as the density character of $E$. (In practice we suppress $\alpha$.) When $n$ is infinite, we take the convention $M_n:=B(H_n)$, where $H_n$ is a Hilbert space of dimension $n$. By naturally identifying
\[
M_n(E)=
\CB(E_\ast,E),
\]
we may equip each $M_n(E)$ with its corresponding point-weak-$\ast$ topology. Note that if $E=M_k$, this is the just the usual weak-$\ast$ topology on $M_n(M_k)\cong M_{nk}$.

\begin{definition}\label{def:nc_convex_set}
We say that a graded subset 
\[K=\coprod_{n\ge 1} K_n \subseteq \calM(E)\]
 is an \textbf{nc convex set} if for every norm-bounded family $(x_i)\in K_{n_i}$ and every family of matrices $\alpha_i\in M_{n_i,n}$ which satisfies
\begin{equation}\label{eq:sum_alpha_i}
\sum_i \alpha_i^\ast \alpha_i = 1_n,
\end{equation}
we have
\begin{equation}\label{eq:nc_convex_combination}
\sum_i \alpha_i^\ast x_i \alpha_i \in K_n.
\end{equation}
Here the sums \eqref{eq:sum_alpha_i} and \eqref{eq:nc_convex_combination} are required to converge in the point-weak-$\ast$ topologies on $M_n$ and $M_n(E)$, respectively. We say in addition that $K$ is a \textbf{compact nc convex set} if each matrix level $K_n$ is point-weak-$\ast$ compact in $M_n(E)$.
\end{definition}

Usually we refer to the sum in \eqref{eq:nc_convex_combination} as an \textbf{nc convex combination} of the points $x_i$. Succinctly, an nc convex set is one that is \emph{closed under nc convex combinations}. It is equivalent to require only that $K$ is closed under direct sums \eqref{eq:sum_alpha_i} in which the $\alpha_i$'s are co-isometries with orthogonal domain projections, and \emph{compressions} \eqref{eq:nc_convex_combination} when there is only one $\alpha_i$, which must be an isometry.

\begin{definition}\label{def:affine}
Let $K$ and $L$ be nc convex sets. A function $a:K\to L$ is an \textbf{affine nc function} if it is graded
\[
a(K_n)\subseteq L_n,\quad\text{ for all }n\ge 1,
\]
and respects nc convex combinations, i.e. whenever $x_i\in K$ are bounded and $\alpha_i$ are scalar matrices of appropriate sizes such that $\sum_i \alpha_i^\ast x_i\alpha_i$, then
\[
a\Big(\sum_i \alpha_i^\ast x_i\alpha_i\Big) =
\sum_i \alpha_i^\ast a(x_i)\alpha_i.
\]
The function $a$ is \textbf{continuous} if each restriction $a\vert_{K_n}$ is point-weak-$\ast$ continuous.
\end{definition}

Kadison's classical representation theorem \cite{kadison1951representation} asserts that the category of \textbf{function systems}, or archimedean order unit spaces, is equivalent to the category of compact convex sets with continuous affine functions as morphisms. The following noncommutative generalization from \cite[Theorem 3.2.5]{davidson2019noncommutative} for compact nc convex sets is essentially equivalent to \cite[Proposition 3.5]{webster1999krein} for compact matrix convex sets.

\begin{theorem}\label{thm:unital_nc_kadison}
The category of unital operator systems with ucp maps as morphisms is contravariantly equivalent to the category of compact nc convex sets with continuous affine nc functions as morphisms. On objects, the essential inverse functors send an operator system $S$ to its nc state space
\[
\calS(S)=
\coprod_{n\ge 1}\{\phi:S\to M_n\mid \phi \text{ is unital and completely positive}\},
\]
and send a compact nc convex set $K$ to the operator system
\[
A(K) =
\{a:K\to \calM=\calM(\C)\mid a \text{ is a continuous affine nc function}\}.
\]
\end{theorem}

The operator system structure and norm on $A(K)$ is pointwise, i.e. one identifies $M_n(A(K))\cong A(K,\calM(M_n))$, and declares a matrix valued affine nc function if it takes positive values at every point. The order unit is the ``constant function" $x\in K_n\mapsto 1_n\in M_n$. Both essential inverse functors act on morphisms by precomposition. That is, if $\pi:S\to T$ is a ucp map between operator systems, then the corresponding map on nc state spaces sends $\rho:T\to M_n$ to $\rho\pi:S\to M_n$. Likewise, if $a:K\to L$ is affine nc, then $f\mapsto f\circ a :A(L)\to A(K)$ is affine nc.

Recently, the second and third authors together with Kim \cite{kennedy2023nonunital} settled the question of Kadison duality for nonunital operator systems. The key challenge is that in the absence of order units, if $S$ is a nonunital operator system then one must remember the whole \textbf{nc quasistate space}
\[
\mathcal{QS}(S) =
\coprod_{n\ge 1}\{\phi:S\to M_n \mid \phi \text{ is contractive and completely positive}\}
\]
and consider \textbf{pointed} affine nc functions which fix the zero quasistates.

\begin{definition}
Let $K$ be a compact nc convex set and fix a distinguished point $z$. We let
\[
A(K,z) =
\{a\in A(K)\mid a(z)=0\}
\]
denote the operator system of affine nc functions which vanish at $z$. We say that the pair $(K,z)$ is a \textbf{pointed nc convex set} if the natural evaluation map
\begin{align*}
K &\to \mathcal{QS}(A(K,z)) \\
x&\mapsto (a\mapsto a(x))
\end{align*}
is surjective (and hence bijective).
\end{definition}

The following result from \cite[Theorem 4.9]{kennedy2023nonunital} provides a nonunital generalization of the duality between the category of unital operator systems and the category of compact nc convex sets. The main subtlety is that while the correspondence $S\mapsto (\mathcal{QS}(S),0)$ is a full and faithful functor, it is only essentially surjective onto the \emph{pointed} compact nc convex sets.

\begin{theorem}\label{thm:nonunital_nc_kadison}
The category of operator systems with ccp maps as morphisms is contravariantly equivalent to the category of pointed compact nc convex sets with pointed continuous affine nc functions as morphism. On objects, the essential inverse functors send an operator system $S$ to is pointed nc quasistate space $(\mathcal{QS}(S),0)$, and send a pointed compact nc convex set $K$ to the operator system $A(K,z)$ of pointed continuous affine nc functions on $(K,z)$.
\end{theorem}

Again, on morphisms the essential inverse functors in Theorem \ref{thm:nonunital_nc_kadison} act in the natural way by precomposition on either affine nc functions or on nc quasistates.

\section{Quotients of matrix ordered spaces}\label{sec:quotients}

\subsection{Operator space quotients}

In this section, if $V$ is a normed vector space and $r>0$, then we let $B_r(V)$ denote the closed ball in $V$ with radius $r$ and center $0$.

Here, we recall the basic theory of quotients for operator spaces. If $E$ is an operator space, and $F\subseteq E$ is a closed subspace, then the quotient vector space $E/F$ is an operator space where the matrix norms isometrically identify $M_n(E/F)$ with the standard Banach space quotient $M_n(E)/M_n(F)$.

\begin{definition}\label{def:operator_space_quotient_map}
Let $\phi:E\to F$ be a completely bounded map between operator spaces $E$ and $F$. We will say that $\phi:E\to F$ is a \textbf{operator space quotient map with constant } $C>0$ if any of the following equivalent conditions hold
\begin{itemize}
\item [(1)] $B_1(M_n(F))\subseteq \overline{\phi_n(B_C(M_n(E)))}=C\overline{\phi_n(B_1(M_n(E)))}$ for all $n\in \N$.
\item [(2)] $B_1(M_n(F))\subseteq (C+\epsilon)\cdot \phi_n(B_1(M_n(E)))$ for all $n\in \N$ and every $\epsilon>0$.
\item [(3)] The induced map $\tilde{\phi}:E/\ker \phi \to F$ is an isomorphism and satisfies $\|\tilde{\phi}^{-1}\|_\text{cb} \le C$.
\end{itemize}
\end{definition}

The equivalence of (1) and (2) follows from a standard series argument using completeness of $E$. We will simply say \textbf{operator space quotient map} if we have no need to refer to $C$ explicitly.

The following fact is standard in operator space theory, but we provide a proof for completeness.

\begin{proposition}\label{prop:operator_space_quotient_dual}
Let $\phi:E\to F$ be a completely bounded map between operator spaces $E$ and $F$. The map $\phi$ is a quotient map with constant $C>0$ if and only if the dual map $\phi^\ast:F^\ast\to E^\ast$ is completely bounded below by $1/C$. Moreover, in this case, $\phi^\ast$ is weak-$\ast$ homeomorphic onto its range.
\end{proposition}

\begin{proof}
Suppose that $C\phi_n(B_1(M_n(E)))$ is dense in $B_1(M_n(F))$ for every $n$. Given $f\in M_m(F^\ast)\cong \CB(F,M_m)$, approximating unit vector $y\in B_1(M_n(F))$ with vectors of the form $\phi(x)$ for $x\in B_C(E)$ shows that $\|f\|_\text{cb}\le C\|\phi^\ast_m(f)\|_\text{cb}$.
\bigskip

Conversely, suppose that
\[
\coprod_{n\ge 1} B_1(M_n(F))\not\subseteq C\coprod_{n\ge 1}\phi_n(B_1(M_n(E))).
\]
By the Effros-Winkler nc Bipolar theorem \cite{effros1997matrix}, there are $m,n\ge 1$, an $x\in CB_1(M_n(E))$, and an $f\in M_m(F^\ast)\cong \CB(F,M_m)$, such that
\[
\Real f_k(y)\le 1_{mk}\quad\text{for all }k\ge 1, y\in B_1(M_k(F)),
\]
and yet $\Real f_n(x)\not\le 1_{mn}$. It follows that $\|f\|\le 1$, but $\|x\|\le C$ and $\|f_n(\phi_n(x))\|>1$, so $\|\phi^\ast_m(f)\|>\|f\|_\text{cb}/C$. This shows $\phi^\ast$ is not completely bounded below by $1/C$.
\bigskip

Finally, if $\phi$ is an operator space quotient map, it is bounded and surjective, and so its dual map $\phi^\ast$ is weak-$\ast$ homeomorphic onto its range.
\end{proof}

%\begin{proof}
%Define the closed nc convex subsets
%\begin{align*}
%K &= \coprod_{n\ge 1} B_1(M_n(F)) \quad\text{and}\\
%L &= \coprod_{n\ge 1} \phi_n(B_1(M_n(E)))
%\end{align*}
%of $\calM(F)$. By the Effros-Winkler nc Bipolar theorem \cite{effros1997matrix}, we have $K\subseteq CL$ if and only if $L^\pi\subseteq CK^\pi$, where $K^\pi$ and $L^\pi$ denote the nc polars
%\begin{align*}
%K^\pi &= \coprod_{m\ge 1}\{f \in M_m(F^\ast) \mid \Real f_n(y) \le 1_{mn} \text{ for all }n\ge 1, y\in K_n\} \\ &=
%\{f\in \calM(F^\ast)\mid \|\Real f\|_\text{cb}\le 1\}
%\end{align*}
%and
%\begin{align*}
%L^\pi &=\coprod_{m\ge 1}\{f \in M_m(F^\ast) \mid \Real f_n(\phi_n(x)) \le 1_{mn} \text{ for all }n\ge 1, x\in B_1(M_n(E))\} \\ &=
%\{
%\end{align*}
%\end{proof}

\subsection{Matrix ordered operator space quotients}

\begin{definition}
Let $X$ be a matrix ordered operator space. We call a closed subspace $J\subseteq X$ a \textbf{kernel} if it is the kernel of a ccp map $\phi:X\to Y$ for some matrix ordered operator space $Y$. In this case, we define an matrix ordered operator space structure on the operator space $X/J$ with involution
\[
(x+J)^\ast :=
x^\ast+J
\]
and matrix order
\[
M_n(X/J)^+:=
\overline{\{x+M_n(J)\mid x\in M_n(X)^+\}},
\]
where the closure is taken in the quotient norm topology on
\[
M_n(X/J)\cong M_n(X)/M_n(J).
\]
\end{definition}

\begin{proposition}
If $X$ is a matrix ordered operator space, and $J=\ker \phi$ is a kernel, then $X/J$ is a matrix ordered operator space.
\end{proposition}

\begin{proof}
Since the involution on $X$ is completely isometric and $J$ is selfadjoint, it follows that the involution on $M_n(X/J)$ is completely isometric. It is straightforward to check that $X/J$ is a matrix ordered operator space. To prove that it is a matrix ordered operator space, suppose $x+J\in M_n(X/J)^+\cap (-M_n(X/J)^+)$. Then for any $\epsilon$, there are $y,z\in M_n(X)^+$ with $\|x-y+M_n(J)\|,\|x+z+M_n(J)\|<\epsilon$. Hence
\[
\|\phi_n(x)-\phi_n(y)\|\le \|x-y+M_n(J)\|< \epsilon
\]
and similarly $\|\phi_n(x)+\phi_n(z)\|<\epsilon$. Since $\phi$ is cp, $\phi_n(y),\phi_n(z)\ge 0$. As $\epsilon$ is arbitrary and $Y$ is a matrix ordered operator space, this shows
\[
\phi_n(x) \in
\overline{M_n(Y)^+}\cap (-\overline{M_n(Y)^+}) =
M_n(Y)^+\cap (-M_n(Y)^+)=
\{0\}.
\]
Therefore $x\in M_n(\ker \phi)=M_n(J)$, and so $x+M_n(J)=0$. This shows 
\[
M_n(X/J)^+\cap (-M_n(X/J)^+)=\{0\},
\]
so $X/J$ is a matrix ordered operator space.
\end{proof}

One can form a category of matrix ordered operator spaces with morphisms as either completely contractive and positive (ccp) or completely bounded and positive (cbp) maps.

\begin{definition}\label{def:MOS_quotient}
Let $X$ and $Y$ be matrix ordered operator spaces, and let $\phi:X\to Y$ be a cbp map. We say that $\phi$ is a \textbf{matrix ordered operator space quotient map with constant} $C>0$ if  for all $n\in \N$ we have both
\begin{itemize}
\item [(1)] $B_1(M_n(Y))\subseteq C\overline{\phi_n(B_1(M_n(X)))}$, and
\item [(2)] $M_n(Y)^+=\overline{\phi_n(M_n(X))^+}.$
\end{itemize}
For brevity, we will usually simply refer to $\phi$ as a \textbf{quotient map}, whenever it is clear that we are speaking only in the context of matrix ordered operator spaces.
\end{definition}

That is, a matrix ordered operator space quotient map is just an operator space quotient map that maps the positives (densely) onto the positives at each matrix level. Comparing to Definition \ref{def:operator_space_quotient_map}.(2), a quotient map is surjective. Each map $\phi_n:M_n(X)\to M_n(Y)$ is therefore open and closed, and since the positive cones $M_n(X)^+$ and $M_n(Y)^+$ are norm-closed, it follows that $\phi_n(M_n(X)^+)$ is closed and $\phi_n(M_n(X)^+)=M_n(Y)^+$ for all $n$. That is, the closure in condition (2) is redundant. The first thing to show is that such maps are in fact categorical quotients in the category of matrix ordered operator spaces.

\begin{proposition}\label{prop:cbp_quotient}
Let $\phi:X\to Y$ be a cbp map between matrix ordered operator spaces. The following are equivalent.
\begin{itemize}
\item [(1)] The map $\phi$ is a quotient map with constant $C>0$.
\item [(2)] The dual map $\phi^\ast:Y^\ast \to X^\ast$ is completely bounded below by $1/C$ and a complete order injection.
\item [(3)] With $J=\ker \phi$, the induced map $\tilde{\phi}:X/J \to Y$ such that
\[\begin{tikzcd}
	X \arrow[r,"\phi"] \arrow[d,"q"] &  Y \\
	X/J \arrow[ur,dashed,swap,"\tilde{\phi}"]
\end{tikzcd}\]
commutes is an isomorphism with cbp inverse satisfying $\|\tilde{\phi}^{-1}\|_\text{cb}\le C$.

\item [(4)] For every matrix ordered operator space $Z$ and cbp map $\psi:X\to Z$ with $\ker\phi\subseteq \ker \psi$, there is a unique cbp map $\tilde{\psi}:Y\to Z$ making the diagram
\[\begin{tikzcd}
	X \arrow[r,"\psi"] \arrow[d,"\phi"] &  Z \\
	Y \arrow[ur,dashed,swap,"\tilde{\psi}"]
\end{tikzcd}\]
commute, with $\|\tilde{\psi}\|_\text{cb}\le C\|\psi\|_\text{cb}$.
\end{itemize}
In this case, $\phi^\ast$ is weak-$\ast$ homeomorphic onto its range.
\end{proposition}

\begin{proof}
To prove (1) and (2) are equivalent, after invoking Proposition \ref{prop:operator_space_quotient_dual}, it suffices to show that $\phi^\ast$ is a complete order injection if and only if Condition (2) in Definition \ref{def:MOS_quotient} holds. Note that because $\phi$ is completely positive, so is $\phi^\ast$. Suppose $\phi_n(M_n(X)^+)$ is dense in $M_n(Y)^+$ for every $n\ge 0$. Let $f\in M_m(Y^\ast)$ with $\phi_m^\ast(f)\ge 0$. Given $n\ge 1$ and $y\in M_n(Y)^+$, approximating $y$ with a net of points of the form $\phi_n(x_i)$ for $x_i\in M_n(X)^+$ shows that
\[
f_n(y)=
\lim_i f_n(\phi_n(x_i))=
\lim_i (\phi^\ast_m (f))_n(x_i)\ge 0.
\]
This shows $f\ge 0$.

Conversely, suppose that $\phi_n(M_n(X)^+)$ is not dense in $M_n(Y)^+$ for some $n\ge 1$. By the Effros-Winkler nc Bipolar Theorem \cite{effros1997matrix} applied to the closed nc convex sets
\[
\coprod_{k\ge 1} \overline{\phi_k(M_n(X)^+)} \quad\not\supseteq\quad \coprod_{k\ge 1} M_k(Y)^+,
\]
there is a selfadjoint matrix functional $f\in M_m(Y^\ast)^\sa$ such that $f_k(y)\ge -1_{mk}$ for every $k$ and every $y\in M_k(Y)^+$, but
\[
f_n(z)\not\ge -1_{mn}
\]
for some $z\in \overline{\phi_n(M_n(X)^+)}\setminus M_k(Y)^+$. A rescaling argument shows that $f\ge 0$ in $M_k(Y)$. However, approximating $x$ by points of the form $\phi_n(x)$, $x\in M_n(X)^+$ shows that $\phi^\ast_m(f)$ cannot be positive. Hence, $\phi^\ast$ is not a complete order isomorphism.
\bigskip

If $\phi$ is a quotient map with constant $C>0$, then it follows immediately from the definition of the matrix order and matrix norms on $X/J$ that $\tilde{\phi}:X/J\to Y$ is a complete order and norm isomorphism with $\|\tilde{\phi}^{-1}\|_\text{cb}\le C$. Conversely, note that by definition the quotient map $q:X\to X/J$ is a quotient map with constant $1$. Hence, if $\tilde{\phi}$ is a complete order isomorphism with $\|\tilde{\phi}^{-1}\|_\text{cb}\le C$, it follows that $\phi=\tilde{\phi}\circ q$ is a quotient map with constant $C$. This proves (1) and (3) are equivalent.
\bigskip

To show (3) and (4) are equivalent, it is enough to note that the quotient map $q:X\to X/J$ satisfies the universal property (4) with constant $C=1$. In detail, if (3) holds, composing the universal map from (4) applied to $q:X\to X/J$ with $\tilde{\phi}^{-1}$ shows that (4) holds for $\phi$ with constant $C$. Conversely, if (4) holds, then it holds for both $\phi$ and $q$, and there are induced maps $\tilde{\phi}:X/J\to Y$ and $\tilde{q}:Y\to X/J$ with $\|\tilde{q}\|\le \|\tilde{\phi}\|_\text{cb}$ and $\|\tilde{q}\|\le C\|q\|_\text{cb}=C$. Comparing diagrams shows $\tilde{q}=\tilde{\phi}^{-1}$, and $\tilde{\phi}$ is an isomorphism.
\end{proof}

Condition (4) in Proposition \ref{prop:cbp_quotient} shows that a matrix ordered operator space quotient map is a categorical quotient in the category of matrix ordered operator spaces with cbp maps as morphisms. Moreover, the norm bound shows that a quotient map with constant $C=1$ is a categorical quotient in the subcategory of matrix ordered operator spaces with ccp maps as morphisms.

\begin{remark}\label{rem:quotients_differ}
Every unital operator system is a matrix ordered operator space, and so if $\phi:S\to T$ is a ucp map between operator systems with $J=\ker S$, we may form the quotient matrix ordered operator space $S/\ker\phi$, but there is no a priori guarantee that this quotient is again an operator system. The matrix ordered operator space quotient is generally \emph{not} isomorphic to the unital operator system quotient defined by Kavruk, Paulsen, Todorov, and Tomforde \cite{kavruk2013quotients}. For example, they show in \cite[Example 4.4]{kavruk2013quotients} that the order norm on the unital operator system quotient need not be completely equivalent to the quotient operator space norm.
\end{remark}

\section{Extension property for compact nc convex sets}\label{sec:extension}

If $K=\coprod_n K_n$ is a compact nc convex set, we will define
\[
\Span_\R K:=
\coprod_{n\ge 1} \Span_\R K_n \subseteq \calM(E).
\]
The set $\Span_\R K$ is also nc convex, but need not be closed in $E$.

\begin{lemma}\label{lem:ncconv_span}
Let $K\subseteq \calM(E)$ be a compact nc convex set that contains $0$. Let $K-K$ denote the levelwise Minkowski difference of $K$ with itself. Then we have inclusions
\[
\frac{K-K}{2}\subseteq
\cncconv(K\cup (-K))\subseteq
K-K.
\]
Consequently, $\cncconv(K\cup (-K))\subseteq \Span_\R K$.
\end{lemma}

\begin{proof}
It is immediate that $(K-K)/2\subseteq \ncconv(K\cup (-K))\subseteq \cncconv (K\cup (-K))$. Given $z\in \ncconv(K\cup (-K))_n$, we can write
\[
z=
\sum_i \alpha_i^\ast x_i\alpha_i -
\sum_j \beta_j^\ast y_j\beta_j
\]
for uniformly bounded families $\{x_i\},\{y_i\}$ in $K$ and matrix coefficients satisfying $\sum_i \alpha_i^\ast \alpha_i + \sum_j \beta_j^\ast\beta_j=1_n$. Since $0\in K$ and $\sum_i \alpha_i^\ast \alpha_i\le 1$, we have $x:=\sum_i\alpha_i^\ast x_i\alpha_i \in K_n$. Similarly $y:=\sum_j \beta_j^\ast y_j\beta_j\in K$, and so $z=x-y$ is in $(K-K)_n=K_n-K_n$. Therefore
\[
\ncconv(K\cup (-K))\subseteq K-K,
\]
and since the latter is compact, $\cncconv(K\cup (-K))\subseteq K-K$.
\end{proof}

When $0\in K$, by extending the inclusion map $K\subseteq \coprod_n M_n(A(K,0)^\ast)$ linearly at each level, we will think of elements in $(\Span_\R K)_n$ as nc functionals in 
\[
M_n(A(K,0)^\ast)=\text{CB}(A(K,0),M_n).
\]

\begin{proposition}\label{prop:span_dual}
Let $0\in K\subseteq \calM(E)$ be a compact nc convex set in a dual operator space $E=(E_\ast)^\ast$. For each $n\in \N$, the inclusion $K\to \text{QS}(A(K,0))$ extends uniquely to a well-defined affine nc isomorphism
\[
\eta:\coprod_{n\ge 1}\Span_\R K_n \to \coprod_{n\ge 1} M_n(A(K,0)^\ast)^\sa
\]
which is levelwise linear. The norm unit ball in $M_n(A(K,0)^\ast)^\sa$ is
\[
B_1(M_n(A(K,0)^\ast)) =
\text{CC}(A(K,0),M_n) =
\cncconv(\eta(K)\cup (-\eta(K)))_n,
\]
and for each $n$, $\eta$ is homeomorphic on $K_n-K_n$.
\end{proposition}

\begin{proof}
Since $K_n$ is convex, we have $\Span_\R K_n=\{sx-ty\mid x,y\in K_n, s,t\ge 0\}$. Given $sx-ty\in \Span_\R K_n$, we define
\[
\eta(sx-ty)(a)=
sa(x)-ta(y)
\]
for $a\in A(K,0)$. Since such functions $a$ are affine and satisfy $a(0)=0$, it follows that $\eta\vert_{K_n}$ is well-defined and linear, and that $\eta$ is affine nc. Since $E_\ast$ contains a separating family of functionals, which restrict to affine nc functions in $A(K,0)$, the map $\eta$ is injective.

Next we will show the closed unit ball is
\[
B_1(M_n(A(K,0)^\ast)^\sa) =
\cncconv(\eta(K)\cup(-\eta(K)))_n
\]
for every $n$. That is, if $L$ is the compact nc convex set
\[
L=
\coprod_{n\ge 1} L_n=
\coprod_{n\ge 1} B_1(M_n(A(K,0)^\ast)^\sa),
\]
we want to show $L= \cncconv(\eta(K)\cup(-\eta(K)))$. Since $\eta(K)$ consists of nc quasistates on $A(K,0)$, it is clear that $L\supseteq \cncconv(\eta(K)\cup(-\eta(K)))$. To prove the reverse inclusion, by the nc Bipolar theorem of Effros and Winkler \cite{effros1997matrix}, it suffices to suppose that for some $n\in \N$ and $a\in M_n(A(K,0))^\sa$ that we have
\[
\phi_n(a)\le
1_k\otimes 1_n =
1_{kn}
\]
for all $k\in \N$ and all $\phi\in \cncconv(\eta(K)\cup(-\eta(K)))$, and then show that $\psi_n(a)\le 1_k\otimes1_n$ for all $k$ and all $\psi\in L_k$. Because $\cncconv(\eta(K)\cup(-\eta(K)))$ contains both $\eta(K)$ and $-\eta(K)$, we have
\[
-1_{kn}\le
a(x)\le
1_{kn}
\]
for all $k$ and all $x\in K_k$. Hence $\|a\|_{M_n(A(K,0))}\le 1$, and so $\psi_n(a)\le \|a\| 1_{kn}\le 1_{kn}$ for every $\psi \in L$. This proves $L=\cncconv(\eta(K)\cup (-\eta(K)))$, and consequently $\eta$ is also surjective. Since $\eta$ is homeomorphic on $K$ and $K-K$ is (levelwise) compact, it is easy to check that $\eta$ is continuous on each $K_n-K_n$. Being a continuous injection on a compact Hausdorff space, the map $\eta\vert_{K_n-K_n}$ is automatically a homeomorphism onto its range.
\end{proof}

Recall that the pair $(K,0)$ in Proposition \ref{prop:span_dual} is a \textbf{pointed nc convex set} exactly when we have
\[
\coprod_{n\ge 1}B_1(M_n(A(K,0)^\ast)^+) =
\text{QS}(A(K,0))=
\eta(K).
\]
In practice, we will often identify $M_n(A(K,0)^\ast)^\sa$ with $\Span_\R K_n$ and so omit the symbol $\eta$. Note that since $\eta$ is homeomorphic on $K-K\supseteq \cncconv(K\cup (-K))$ (Lemma \ref{lem:ncconv_span}), we are free to identify
\[
\cncconv(\eta(K)\cup(-\eta(K))) =
\eta(\cncconv(K\cup(-K))).
\]
That is, when we identify $M_n(A(K,0)^\ast)^\sa=\Span_\R K_n$, the unit ball of $M_n(A(K,0)^\ast)^\sa$ is $\cncconv(K\cup(-K))_n$.

For a closed convex set $X$ in a vector space $V$ containing $0$, we use the usual Minkowski functional
\[
\gamma_X(v):=
\inf\{t\ge 0 \mid v\in t X\},\qquad v\in V.
\] 
If $0\in K=\bigcup_n K_n$ is a compact nc convex set over a dual operator space $E$, we will use the shorthand
\[
\gamma_K(x)=
\gamma_{K_n}(x)
\]
when $x\in M_n(E)$.

\begin{definition}\label{def:width}\cite{taylor1972extension}
If $X$ is a closed convex set in some vector space $V$, then for $d\in V$, we define the \textbf{width of $V$ (with respect to $d$)} or the \textbf{$d$-width of $V$} as
\begin{align*}
|X|_d &:=
\sup \{t\ge 0\mid td\in X-X\} \\ &=
\frac{1}{\gamma_{X-X}(d)}.
\end{align*}
\end{definition}

\begin{definition}
If $K=\coprod_n K_n\subseteq \calM(E)$ is a closed nc convex set over a dual operator space $E$, then for any $n$ and any $d\in M_n(E)$ we define the \textbf{width}
\[
|K|_d:=
|K_n|_d =
\frac{1}{\gamma_{K-K}(d)},
\]
\end{definition}

\begin{lemma}\label{lem:width_dual_norm}
If $0\in K\subseteq \calM(E)$ is a compact nc convex set containing $0$, then for $d\in M_n(E)$, we have $|K|_d>0$ if and only if $d\in \Span_\R K$. Moreover, for $d\in \Span_\R K$, we have
\[
\frac{1}{|K|_d}\le\|\eta(d)\|_{M_n(A(K,0)^\ast)}\le
\frac{2}{|K|_d}.
\]
That is, $d\mapsto 1/|K|_d=1/|K_n|_d$ defines a norm on $\Span_\R K_n$ that is equivalent to the norm induced by the isomorphism $\eta:\Span_\R K_n\to M_n(A(K,0)^\ast)^\sa$.
\end{lemma}

\begin{proof}
By Lemma \ref{lem:ncconv_span}, we have inclusions
\[
\frac{K-K}{2}\subseteq
\cncconv(K\cup (-K))\subseteq
K-K.
\]
It follows that for $d\in \Span_\R K$, we have
\[
2\gamma_{K-K}(d)\ge
\gamma_{\cncconv(K\cup(-K))}(d)\ge
\gamma_{K-K}(d).
\]
By definition, $\gamma_{K-K}=1/|K|_d$. By Proposition \ref{prop:span_dual}, the norm unit ball of $M_n(A(K,0)^\ast)^\sa$ is 
\[
\cncconv(\eta(K)\cup(-\eta(K))) =
\eta(\cncconv(K\cup(-K))),
\]
and hence $\gamma_{\cncconv(K\cup(-K))}(d)=\gamma_{\cncconv(\eta(K)\cup(-\eta(K)))}(\eta(d))=\|\eta(d)\|$.
\end{proof}

Given compact nc convex sets $0\in L\subseteq K$. The restriction map $\rho:A(K,0)\to A(L,0)$ is always completely contractive and positive, and has dense range. When is this map an operator space quotient map? Equivalently, this means there is a constant $C>0$ so that any affine nc function $g\in M_n(A(L,0))$ extends to an affine nc function $f$ on all of $K$ with
\[
f\vert_L=g\quad\text{and}\quad
\|f\|_{M_n(A(K,0))}\le
C \|g\|_{M_n(A(L,0))}.
\]
Here is a noncommutative version of \cite[Theorem 1]{taylor1972extension}.

\begin{proposition}\label{prop:restriction_norm_quotient}
Let $0\in L\subseteq K\subseteq \calM(E)$ be compact nc convex sets containing $0$. The following are equivalent
\begin{itemize}
\item [(1)] The restriction map $A(K)\to A(L)$ is an operator space quotient map.
\item [(2)] The restriction map $\rho:A(K,0)\to A(L,0)$ is an operator space quotient map.
\item [(3)] The dual map $\rho^\ast:A(L,0)^\ast \to A(K,0)^\ast$ is completely bounded below.
\item [(4)] There is a constant $c>0$ such that for all $n\ge 1$ and all $d\in M_n(E)$ with $|L|_d> 0$, we have
\[
|L|_d\ge 
c|K|_d.
\]
\item [(5)] There is a constant $C>0$ such that
\[
(K-K)\cap \Span_\R L \subseteq C(L-L).
\]
\end{itemize}
\end{proposition}

\begin{proof}
Clearly (1) implies (2). Suppose $\rho:A(K,0)\to A(L,0)$ is an operator space quotient map with constant $C\ge 0$. Given $a\in A(L)$, we have $a-a(0)\otimes 1_{A(L)}\in A(L,0)$. Thus there is a $b\in A(K,0)$ with $b\vert_L=a-a(0)\otimes 1_{A(L)}$ and $\|b\|\le C\|a-a(0)\otimes 1_{A(L)}\|\le 2C\|a\|$. Then, $b+a(0)\otimes 1_{A(K)}\in A(K)$ restricts to $a$ on $L$ and satisfies $\|b+a(0)\otimes 1_{A(K)}\|\le \|b\|+\|a\|\le (2C+1)\|a\|$. This proves $A(K)\to A(L)$ is an operator space quotient map with constant $2C+1$, so (2) implies (1).
\bigskip

The equivalence of (2) and (3) is Proposition \ref{prop:operator_space_quotient_dual}. To prove (3) is equivalent to (4), first note by taking real and imaginary parts that (3) occurs if and only if the restrictions $\rho^\ast_n:M_n(A(L,0)^\ast)^\sa \to M_n(A(K,0)^\ast)^\sa$ are bounded below by a universal constant. By Proposition \ref{prop:span_dual}, we may identify
\[
\Span_\R L_n = M_n(A(L,0)^\ast)^\sa \quad\text{and}\quad
\Span_\R K_n = M_n(A(K,0)^\ast)^\sa.
\]
With this identification, $\rho^\ast$ is just the inclusion map $\Span_\R L_n\to \Span_\R K_n$. By Lemma \ref{lem:width_dual_norm}, the induced norms on $\Span_\R L$ and $\Span_\R K$ are completely equivalent to $d\mapsto 1/|L|_d$ and $d\mapsto 1/|K|_d$. Thus the dual map $\rho^\ast$ is completely bounded below if and only if for some constant $c>0$, we have
\[
\frac{1}{c|K|_d}\le
\frac{1}{|L|_d}\iff
|L|_d\ge c|K|_d
\]
whenever $d \in \Span_\R L =\{d\in \calM(E)\mid |L|_d>0\}$, by Lemma \ref{lem:width_dual_norm}.
\bigskip

For $d\in \calM(E)$, recall that $|K|_d =\frac{1}{\gamma_{K-K}(d)}$ and $|L|_d=\frac{1}{\gamma_{L-L}(d)}$. Hence condition (3) holds if and only if
\[
\gamma_{L-L}\vert_{\Span_\R L} \le
\frac{1}{c}\gamma_{K-K}\vert_{\Span_\R L} =
\gamma_{c(K-K)}\vert_{\Span_\R L}.
\]
Using only the definition of the Minkowski gauges $\gamma_{K-K}$ and $\gamma_{L-L}$, this holds if and only if
\[
c(K-K)\cap \Span_\R L\subseteq L-L.
\]
Hence condition (4) holds with constant $c>0$ if and only if condition (5) holds with constant $C=1/c>0$.
\end{proof}

Note that for any general inclusion $L\subseteq K$ of compact nc convex sets, we can freely translate to assume $0\in L$ and apply Proposition \ref{prop:restriction_norm_quotient}. Thus conditions (1), (4), and (5) are equivalent in total generality. Note also that we do not require in \ref{prop:restriction_norm_quotient} that $(L,0)$ and $(K,0)$ are \emph{pointed} nc convex sets.

\begin{example}
It is possible that the restriction map $A(K,0)\to A(L,0)$ in Proposition \ref{prop:restriction_norm_quotient} is surjective but not an operator space quotient. For instance, let $E$ be an infinite dimensional Banach space. Let $\max(E)$ and $\min(E)$ denote $E$ equipped with its maximal and minimal operator space norms which restrict to the usual norm on $E$ \cite[Section 3.3]{effros2022theory}. There are standard operator space dualities $\max(E)^\ast\cong \min(E^\ast)$ and $\min(E)^\ast\cong \max(E^\ast)$. As $E$ is infinite dimensional, the maximal and minimal matrix norms on $E$ are not completely equivalent \cite[Theorem 14.3]{paulsen2002completely}. So, the identity map $\max(E)\to \min(E)$ is surjective and not an operator space quotient map. Consider the minimal and maximal nc unit balls
\[
K=\coprod_{n\ge 1} B_1(M_n(\min (E^\ast))) \quad\text{and}\quad
L=\coprod_{n\ge 1} B_1(M_n(\max (E^\ast)))
\]
in $\calM(E^\ast)$. By the dualities $\max(E)^\ast\cong \min(E^\ast)$ and $\min(E)^\ast\cong \max(E^\ast)$, we have
\[
A(K,0) \cong \max(E) \quad\text{and}\quad A(L,0)\cong \min(E)
\]
completely isometrically. The restriction map $A(K,0)\to A(L,0)$ is just the identity map $\max(E)\to \min(E)$, which is surjective, but not an operator space quotient map.
\end{example}

\begin{example}
Proposition \ref{prop:restriction_norm_quotient} provides a guarantee that every matrix-valued affine nc function on $L$ lifts to an affine nc function on $K$ with a complete norm bound. However, there is no guarantee that we can lift a \emph{positive} affine function to one that is positive. For instance, the restriction map of function systems
\[
A([-1,1],0)\to A([0,1],0)
\]
is an operator space quotient map with constant $c=1$, but does not map the positives onto the positives because $A([-1,1],0)^+=\{0\}$.
\end{example}

\begin{proposition}\label{prop:restriction_order_quotient}
Let $0\in L\subseteq K\subseteq \calM(E)$ be compact nc convex sets such that $(L,0)$ and $(K,0)$ are pointed compact nc convex sets. Let $\rho:A(K,0)\to A(L,0)$ be the restriction map. The following are equivalent
\begin{itemize}
\item [(1)] For all $n\ge 1$, $\overline{\rho_n(M_n(A(K,0)^+))}=M_n(A(L,0))^+$.
\item [(2)] The dual map $\rho^\ast:A(L,0)^\ast\to A(K,0)^\ast$ is a complete order embedding.
\item [(3)] $K\cap \Span_\R L\subseteq \R_+ L$.
\item [(4)] $K\cap \cncconv(L\cup (-L))=L$.
\end{itemize}
\end{proposition}

\begin{proof}
To prove $(1)\iff (2)$, consider the closed nc convex sets
\[
P=\coprod_{n\ge 1}M_n(A(L,0))^+ \quad\text{and}\quad
Q=\coprod_{n\ge 1} \overline{\rho_n(M_n(A(K,0))^+)}.
\]
By the nc Bipolar theorem of Effros and Winkler \cite{effros1997matrix}, we have $Q=P$ if and only if their nc polars $Q^\pi$ and $P^\pi$ are equal. Since each level $P_n$ is a convex cone, a standard scaling argument shows that
\begin{align*}
P^\pi &=
\{\phi \in M_k(A(L,0)^\ast)\mid k\in \N,\real \phi_n(b)\le 1_{nk} \text{ for all } n\ge 1, b\in P_n\} \\ &=
\{\phi \in M_k(A(L,0)^\ast) \mid k\in \N,\real \phi\le 0\}
\end{align*}
and similarly
\begin{align*}
Q^\pi &=
\{\phi \in M_k(A(L,0)^\ast) \mid k\in \N,\real \phi_n(\rho_n(a))\le 1_{nk} \text{ for all }n\ge 1, a\in M_n(A(K,0))^+\} \\ &=
\{\phi\in M_k(A(L,0)^\ast) \mid k\in \N,\real \rho^\ast_k(\phi)\le 0\}.
\end{align*}
Thus $P=Q$ if and only if $\rho^\ast$ is a complete order injection.
\bigskip

When we identify $A(K,0)^\ast=\Span_\R K_1$ and $A(L,0)^\ast=\Span_\R L_1$ as in Proposition \ref{prop:span_dual}, the dual map $\rho^\ast:\Span_\R L\to \Span_\R K$ is just the inclusion map. Since $(K,0)$ and $(L,0$) are pointed, the positive cones in $M_n(A(K,0)^\ast)=\Span_\R K_n$ and $M_n(A(L,0)^\ast) = \Span_\R L_n$ are just $\R_+ K_n$ and $\R_+ L_n$, respectively. Hence the inclusion map is a complete order injection if and only if we have
\[
\R_+ K\cap \Span_\R L = \R_+ L.
\]
A rescaling argument shows that this is equivalent to
\[
K\cap \Span_\R L \subseteq \R_+ L,
\]
and so (2) and (3) are equivalent.
\bigskip

If $K\cap \cncconv(L\cup (-L))=L$, then scaling gives
\[
\R_+(K\cap \Span_\R L) = 
\R_+ K\cap \Span_\R L = \R_+L,
\]
which is again equivalent to (3), so (4) implies (3). Now suppose that $K\cap \Span_\R L\subseteq \R_+ L$. Clearly $L\subseteq K\cap \cncconv(L\cup (-L))$. Conversely, if $x\in K\cap \cncconv(L\cup (-L))$, then by Lemma \ref{lem:ncconv_span}, we also have $x\in K\cap \Span_\R L=\R_+ L$. Hence
\[
x\in \cncconv(L\cup (-L))\cap \R_+L.
\]
Because $(L,0)$ is pointed, this implies $x\in L$, proving that (3) implies (4).
\end{proof}

Combining Propositions \ref{prop:restriction_norm_quotient} and \ref{prop:restriction_order_quotient} yields

\begin{theorem}\label{thm:dualizable_extrinsic}
Let $(L,0)$ and $(K,0)$ be pointed compact nc convex sets with $L\subseteq K\subseteq \calM(E)$. The following are equivalent.
\begin{itemize}
\item [(1)] The restriction map $A(K,0)\to A(L,0)$ is a matrix ordered operator space quotient map.
\item [(2)] There is a constant $C>0$ such that
\begin{itemize}
\item [(i)]  $(K-K)\cap \Span_\R L\subseteq C(L-L)$, and
\item [(ii)] $K\cap \Span_\R L\subseteq \R_+ L$.
\end{itemize}
%\item [(3)] \todo{(I am not sure if this condition (3) works yet.)} There is a constant $C>0$ such that
%\[
%K\cap \Span_\R L \subseteq CL.
%\]
\end{itemize}
\end{theorem}

%%%Proof for tentative item (3) which isn't working yet.
%\begin{proof} 
%The equivalence of (1) and (2) follows from combining Propositions \ref{prop:restriction_norm_quotient} and \ref{prop:restriction_order_quotient}.
%\bigskip
%
%Suppose that conditions (i) and (ii) hold as in (2). Without loss of generality, we may assume $C\ge 1$. Note that $K\cap \Span_\R L$ is contained in both $(K-K)\cap \Span_\R L$ and $\R_+L$. Thus for $x\in K\cap \Span_\R L$, we have $x\in C(L-L)\cap \R_+L$. It follows that $x/2C\in ((L-L)/2)\cap \R_+L$. Since $(L,0)$ is pointed, we have
%\[
%\frac{x}{2C}\in \frac{L-L}{2}\cap \R_+L\subseteq
%\cncconv(L\cap (-L))\cap \R_+L=L.
%\]
%This proves that $K\cap \Span_\R L\subseteq 2CL$.
%
%\bigskip
%
%Now suppose that $K\cap \Span_\R L \subseteq CL$. Again we are free to assume that $C\ge 1$. Then clearly condition (ii) in (2) holds. Given $x\in (K-K) \cap\Span_\R L$, we may write $x=t(y-z)$ for some $y,z\in L$ and $t\ge 0$. Then $(x+tz)/(1+t)=ty$ is in $K\cap \Span_\R L$, and hence in $CL$. Therefore
%\[
%x=
%(1+t)\left(\frac{x+tz}{1+t}-\frac{t}{1+t}z\right)
%\]
%lies in $(1+t)C(L-L)$. \todo[inline]{Problem: $t$ depends on $x$. We would need to bound $t$ in terms of $\|x\|_{A(K,0)^\ast}\sim \gamma_{K-K}(x)$, but such a bound would be tantamount to (i) already...}
%\end{proof}

\section{Dualizability via nc quasistate spaces}\label{sec:dualizability}

Recall that the trace class operators $\calT(H)=B(H)_\ast$ inherit a matrix ordered operator space structure via the embedding $\calT(H)=B(H)_\ast\subseteq B(H)^\ast$, where $B(H)\cong (B(H)_\ast)^\ast$ completely isometrically and order isomorphically. By Ng's \cite{ng2022dual} results, since $B(H)$ is a C*-algebra, $B(H)^\ast$ is an operator system, and so $\calT(H)=B(H)_\ast\subseteq B(H)^\ast$ is also an operator system. The nc quasistate space of $\calT(H)$ is the compact nc convex set
\[
\calP(H):=
\coprod_n M_n(B(H))_1^+ =
\coprod_n \{x\in M_n(B(H))\vert x\ge 0, \|x\|\le 1\}.
\]
Applying Theorem \ref{thm:dualizable_extrinsic} and Proposition \ref{prop:cbp_quotient} yields the following extrinsic geometric characterization of dualizability for an operator system.

\begin{corollary}\label{cor:dualizable_extrinsic}
Let $S$ be an operator system with pointed nc quasistate space $(K,0)$, and let $H$ be a Hilbert space. The following are equivalent.
\begin{itemize}
\item [(1)] There is a weak-$\ast$ homeorphic complete embedding $S^\ast \to B(H)$.
\item [(2)] There is a matrix ordered operator space quotient map $\calT(H)\to S$.
\item [(3)] There is a pointed continuous affine nc injection $\phi:(K,0)\to \calP(H)$ such that
\begin{itemize}
	\item [(i)] $(\calP(H)-\calP(H))\cap \Span_\R \phi(K) \subseteq C(\phi(K)-\phi(K))$ for some constant $C>0$, and
	\item [(ii)] $\calP(H)\cap \Span_\R \phi(K)\subseteq \R_+\phi(K)$.
\end{itemize}
\end{itemize}
\end{corollary}

\begin{definition}\label{def:alpha_generated}
Let $E$ be an ordered $\ast$-Banach space with closed positive cone $E^+$. We say $E$ is \textbf{$\alpha$-positively generated} or simply \textbf{$\alpha$-generated} for a constant $\alpha>0$ if for each $x\in E^\sa$, we can write
\[
x=y-z
\]
for $y,z\in E^+$ satisfying $\|y\|+\|z\|\le \alpha\|x\|$. Or, equivalently,
\[
B_1(E)=\alpha\conv(B_1(E^+)\cup (-B_1(E^+))).
\]
If $X$ is a matrix ordered operator space, then we say $X$ is \textbf{completely $\alpha$-generated} if each matrix level $M_n(X)$ is $\alpha$-generated.
\end{definition}

In \cite[Theorem 3.9]{ng2022dual}, Ng proved that an operator system $S$ is dualizable if and only if it is completely $\alpha$-generated for some $\alpha>0$. The following definition is the dual property of $\alpha$-generation.

\begin{definition}\label{def:alpha_normal}
An ordered $\ast$-Banach space $E$ is \textbf{$\alpha$-normal} for some $\alpha>0$ if for all $x,y,z\in E^\sa$,
\begin{equation}\label{eq:alpha_normal}
x\le y \le z \implies
\|y\|\le \alpha \max\{\|x\|,\|z\|\}.
\end{equation}
If $X$ is a matrix ordered operator space, then $X$ is \textbf{completely $\alpha$-normal} if each matrix level $M_n(X)$ is $\alpha$-normal.
\end{definition}

The condition of $\alpha$-normality can be viewed as a strict requirement about how the norm and order structure on $E$ interact. Normality means that ``order bounds" $x\le y\le z$ should imply ``norm bounds" $\|y\|\le \alpha \max\{\|x\|,\|z\|\}$. If one does not care about the exact value of $\alpha$, it is enough to check the normality identity \eqref{eq:alpha_normal} on positive elements in the special case $x=0$.

\begin{proposition}\label{prop:alpha_normal_positive}
If $E$ is an ordered $\ast$-Banach space, then $E$ is $\alpha$-normal for some $\alpha>0$ if and only if there is a constant $\beta>0$ such that
\begin{equation}\label{eq:alpha_normal_positive}
0\le x\le y \implies
\|x\|\le\beta\|y\|
\end{equation}
for $x,y\in E^+$.
\end{proposition}

\begin{proof}
If $E$ is $\alpha$-normal, then \eqref{eq:alpha_normal_positive} holds with $\beta=\alpha$. Conversely, suppose \eqref{eq:alpha_normal_positive} holds, and let $x\le y \le z$ in $E^\sa$. Then $0\le y-x\le z-x$, and so $\|y-x\|\le \beta\|z-x\|$. Then, we get the bound
\begin{align*}
\|y\| &\le
\|y-x\|+\|x\| \\&\le
\beta\|z-x\|+\|x\| \\ &\le
\beta(\|z\|+\|x\|)+\|x\| \\&\le 
(2\beta+1)\max\{\|x\|,\|z\|\},
\end{align*}
proving $E$ is $(2\beta+1)$-normal.
\end{proof}

\begin{proposition}\label{prop:alpha_generated_dual}
Let $X$ be a matrix ordered operator space, with dual matrix ordered operator space $X^\ast$, and let $\alpha>0$. If $X$ is completely $\alpha$-generated, then $X^\ast$ is completely $2\alpha$-normal. Conversely, if $X^\ast$ is completely $\alpha$-normal, then $X$ is completely $2\alpha$-generated.
\end{proposition}

\begin{proof}
Suppose that $X$ is completely $\alpha$-generated. Let $k\in \N$ and suppose $x,y,z\in M_k(X^\ast)^\sa$ satisfy $x\le y \le z$ in the dual matrix ordering on $X^\ast$. By definition of the dual norm, we have
\[
\|y\|_{M_k(X^\ast)} =
\sup\{\|\llangle a,x\rrangle\|\,\mid\, n\ge 1, a\in M_n(X)^\sa\},
\]
where $\llangle\cdot,\cdot\rrangle$ denotes the matrix pairing between $\calM(X)$ and $\calM(X^\ast)$ defined by
\[
M_m(X) \times M_n(X^\ast) \to M_{m \times n} : (a,x) \to \llangle a, x \rrangle = [ x_{k,l}(a_{i,j}) ].
\]

Given $n\in \N$ and $a\in M_n(X)^\sa$, we can write $a=b-c$ where $b,c\in M_n(X)^+$ satisfy $\|b\|+\|c\|\le \alpha \|a\|$. Then, we have the operator inequality
\begin{align*}
\llangle a,y\rrangle  &=
\llangle b,y\rrangle - \llangle c,y\rrangle \\ &\le
\llangle b,z\rrangle - \llangle c,x\rrangle \\ &\le
(\|z\|\|b\|+\|x\|\|c\|)1_{nk} \\ &\le
(\|x\|+\|z\|)\alpha \|a\|1_{nk}.
\end{align*}
Symmetrically,
\begin{align*}
\llangle a,y\rrangle &\ge
\llangle b,x\rrangle - \llangle c,z\rrangle \\&\ge 
-(\|x\|\|b\|+\|z\|\|c\|)1_{nk}\\ &\ge
-(\|x\|+\|z\|)\alpha \|a\|1_{nk}.
\end{align*}
It follows that
\[
\|\llangle a,y\rrangle \|\le
(\|x\|+\|z\|)\alpha \|a\|.
\]
Since $a$ was arbitrary, this shows $\|y\|\le \alpha (\|x\|+\|z\|)\le 2\alpha \max\{\|x\|,\|z\|\}$, proving $X^\ast$ is completely $2\alpha$-normal.
\bigskip

Now suppose $X^\ast$ is completely $2\alpha$-normal. Consider the closed matrix convex subsets
\begin{align*}
K &:= 
\coprod_{n\ge 1} B_1(M_n(X)^\sa)=B_1(\calM(X)^\sa), \\
K^+ &:=
\coprod_{n\ge 1} B_1(M_n(X)^+) =
K\cap \calM(X)^+, \\ 
L &:=
\cncconv \left(K^+ \cup (-K^+)\right)
\end{align*}
of $\calM(X)$. We will show that $K\subseteq \alpha L$.

To prove $K\subseteq \alpha L$, by the selfadjoint version of the nc separation Theorem of Effros and Winkler \cite[Theorem 2.4.1]{davidson2019noncommutative}, it suffices to show that the selfadjoint nc polars 
\[
K^\rho:=
\coprod_{n\ge 1} \{x\in M_n(X)^\sa \mid \llangle a,x\rrangle \le 1_{nk} \text{ for all }k\ge 1,x\in K_k\}
\]
 and $L^\rho$ (defined similarly) satisfy $L^\rho \subseteq \alpha K^\rho$. The relevant selfadjoint polars are
\begin{align*}
K^\rho &=
\coprod_{k\ge 1} B_1(M_k(X^\ast)), \\
(K^+)^\rho &=
K^\rho - \calM(X^\ast)^+ \\ &=
\coprod_{k\ge 1} \{x\in M_k(X^\ast)^\sa\mid x\le y \text{ for some }y\in K^\rho\},\quad \text{ and } \\
L^\rho &=
(K^+)^\rho \cap (-K^+)^\rho \\ &=
(K^\rho - \calM(X^\ast)^+) \cap (K^\rho + \calM(X^\ast)^+) \\ &=
\coprod_{k\ge 1} \{y\in M_k(X^\ast)^\sa \mid x\le y \le z \text{ for some }x,z\in K^\rho\}.
\end{align*}
Hence, if $y\in L^\rho_k$, then $y$ satisfies $x\le y \le z$ for some $x,z\in M_k(X^\ast)^+$ with $\|x\|,\|z\|\le 1$. By complete $\alpha$-normality, this implies $\|y\|\le \alpha$, so $y\in \alpha K^\rho$. This proves $L^\rho \subseteq \alpha K^\rho$, so $K\subseteq \alpha L$.

Hence $K\subseteq \alpha L=\alpha \cncconv(K^+\cup (-K^+))$. Using Lemma \ref{lem:ncconv_span}, we have
\[
\cncconv(K^+\cup (-K^+)) \subseteq
K^+-K^+.
\]
Hence $K\subseteq \alpha (K^+-K^+)$, and by rescaling every element $x\in \calM(X)^\sa$ can be decomposed as $x=y-z$ with $y,z\ge 0$ and $\|y\|,\|z\|\le \alpha\|x\|$, and so $\|y\|+\|z\|\le 2\alpha \|x\|$. This shows that $X$ is completely $2\alpha$-generated.
\end{proof}

\begin{remark}\label{rem:operator_space_normal}
If $H$ is a Hilbert space, then $B(H)$ is completely $1$-normal. Consequently, if $X$ is a matrix ordered operator space which is completely norm and order isomorphic to a subspace of $B(H)$ (a quasi-operator system), then $X$ must be $\alpha$-normal for some $\alpha>0$.
\end{remark}

\begin{remark}\label{rem:normality_hahn_wittstock}
The dual of a C*-algebra is $1$-generated, which in the commutative case $C(X)^\ast$ reflects the usual Hahn-Jordan decomposition of measures. In \cite[Theorem 2.15]{connes2022tolerance}, Connes and van Suijlekom prove that this Hahn-Jordan decomposition extends even to $M_n(S)$, for approximately unital operator systems $S$. 

Complete $1$-generation asserts that this kind of Hahn-Jordan decomposition holds at all matrix levels of $M_n(S^\ast)=\text{CB}(S,M_n)$. Wittstock's Decomposition Theorem shows that if $A$ is a C*-algebra, then $A$ is completely $1$-generated, and by Wittstock's Extension Theorem, this passes to unital operator subsystems of $A$. For a standard reference, see \cite[Theorems 8.2 and 8.5]{paulsen2002completely}.
\end{remark}

Because complete $\alpha$-normality is dual to complete $\alpha$-generation, \cite[Theorem 3.9]{ng2022dual} can be viewed as a partial converse to Remark \ref{rem:operator_space_normal}. If $X=S^\ast$ is the dual of an operator space, then if $X$ is completely $\alpha$-normal, then it is a dual quasi-operator system. Translating the normality condition into a condition on the nc quasistate space gives the following intrinsic characterization of dualizability.

\begin{theorem}\label{thm:dualizable_intrinsic}
Let $(K,0)$ be a pointed compact nc convex set, with associated operator system $S=A(K,0)$. The following are equivalent.
\begin{itemize}
\item [(1)] $S^\ast$ is a dual quasi-operator system.
\item [(2)] $S$ is completely $\alpha$-generated for some $\alpha>0$.
\item [(3)] $S^\ast$ is completely $\alpha$-normal for some $\alpha>0$.
\item [(4)] There is a constant $C>0$ such that
\[
(K-\R_+K)\cap \R_+ K\subseteq CK,
\]
where $K-\R_+K$ denotes the levelwise Minkowski difference.
\item [(5)] The closed nc convex set $(K-\R_+K)\cap \R_+ K$ is bounded.
\end{itemize}
\end{theorem}

\begin{proof}
The equivalence of (1) and (2) was proved by Ng in \cite[Theorem 3.9]{ng2022dual}. Proposition \ref{prop:alpha_generated_dual} shows that (2) and (3) are equivalent. To prove that (3) and (4) are equivalent, we may use Proposition \ref{prop:span_dual} to identify $\coprod_{n\ge 1} M_n(S^\ast)^\sa = \Span_\R K$. After doing so, the positive elements in $\calM(S^\ast)$ correspond to the closed nc convex set $\R_+ K$, and for $d\in \R_+K_n$, we have $\|d\|_{M_n(S^\ast)}=\gamma_K(d)$. Consequently,
\begin{align*}
(K-\R_+ K)\cap \R_+ K &=
\{d\in \Span_\R K\mid 0\le d \le x \text{ for some }x\in K\}\\ &=
\{d\in \coprod_n M_n(S^\ast)^\sa\mid 0\le d \le x \text{ for some }x>0 \text{ in } K_n \text{ with }\|x\|\le 1\}.
\end{align*}
Thus (4) holds if and only if
\[
0\le x\le y \text{ and } \|y\|\le 1 \implies
\|x\|\le C,
\]
in $M_n(S^\ast)^\sa$ for all $n\in \N$. By rescaling, this is equivalent to asserting that
\[
0\le x\le y\implies \|x\|\le C\|y\|
\]
in $M_n(S^\ast)^\sa$. Then, Proposition \ref{prop:alpha_normal_positive} shows that if (3) holds, then (4) holds with $C=\alpha$, and if (4) holds, then (3) holds with $\alpha = 2C+1$. Finally, because $(K-\R_+K)\cap \R_+K$ is a subset of $\R_+ K$, on which the matrix norms from $S^\ast$ agree with the Minkowski gauge $\gamma_K$, (4) holds if and only if $(K-\R_+K)\cap \R_+K$ is bounded by $C>0$, i.e. if and only if (5) holds.
\end{proof}

Note that ``bounded'' in Theorem \ref{thm:dualizable_intrinsic}.(5) is in reference to the system of matrix norms on $\coprod_n M_n(S^\ast)$, i.e. uniform boundedness in cb-norm at each level.

\begin{remark}\label{rem:classical_dualizable}
The analogous version of Theorem \ref{thm:dualizable_intrinsic} holds in the classical case: If $(K,0)$ is a pointed compact convex set, then the nonunital function system $A(K,0)$ is $\alpha$-generated for some $\alpha>0$ if and only if $(K-\R_+ K)\cap \R_+ K$ is bounded.
\end{remark}

\begin{corollary}\label{cor:subset_dualizable}
Let $z\in K\subseteq L$ be compact nc convex sets such that $(K,z)$ and $(L,z)$ are pointed. If $A(L,z)$ is dualizable, then so is $A(K,z)$.
\end{corollary}

\begin{proof}
By translating, it suffices to consider this when $z=0$. This follows by noting that
\[
(K-\R_+K)\cap \R_+K\subseteq (L-\R_+L)\cap \R_+L,
\]
and using condition (5) in Theorem \ref{thm:dualizable_intrinsic}.
\end{proof}

In \cite[Section 8]{kennedy2023nonunital}, quotients of (nonunital) operator systems were defined. There, a quotient of operator systems $S\to S/J$ corresponds dually to a restriction map $A(K,z)\to A(M,z)$ between pointed compact nc convex sets, where $M\subseteq K$ is the annihilator of the kernel $J\subseteq K$. Applying Corollary \ref{cor:subset_dualizable} gives

\begin{corollary}\label{cor:quotient_dualizable}
If $S$ is a dualizable operator system, then every quotient of $S$ is dualizable.
\end{corollary}

\section{Noncommutative simplices} \label{sec:nc-simplices}

A \textbf{noncommutative (Choquet) simplex} $K$ is a compact nc convex set such that every point $x\in K$ has a \emph{unique} representing ucp map on the maximal C*-algebra $C(K)\cong C^\ast_{\max}(A(K))$, which must be the point evaluation $\delta_x$. In \cite{kennedy2022noncommutative}, the second author and Shamovich characterized nc simplices as corresponding dually to unital \textbf{C*-systems} in the sense of Kirchberg and Wassermann \cite{kirchberg1998c}, as follows.

\begin{theorem}\label{thm:unital_nc_simplex}\cite[Theorems 4.7 and 6.2]{kennedy2022noncommutative}
Let $K$ be a compact nc convex set. Then $K$ is an nc simplex if and only if the bidual $A(K)^{\ast\ast}$ is unital completely order isomorphic to a C*-algebra. Moreover, if this is the case, the inclusion $A(K)\hookrightarrow A(K)^{\ast\ast}$ extends to a $\ast$-homomorphism $C(K)\to A(K)^{\ast\ast}$, which further extends to a normal conditional expectation of $C(K)^{\ast\ast}\cong B(K)$ onto $A(K)^{\ast\ast}$.
\end{theorem}

In fact, if $K$ is an nc simplex, then we will need to identify the C*-algebra $A(K)^{\ast\ast}$ as the bidual of the C*-envelope.

\begin{lemma}\label{lem:double_dual_envelope}
Let $K$ be a compact nc simplex. Then $A(K)^{\ast\ast}$ is $\ast$-isomorphic to the bidual $\Cmin(A(K))^{\ast\ast}$ via a $\ast$-isomorphism preserving $A(K)$.
\end{lemma}

\begin{proof}
Included in \cite[Theorem 4.7]{kennedy2022noncommutative} is the fact that the $\ast$-homomorphism $C(K)\to A(K)^{\ast\ast}$ preserving $A(K)$ factors through the C*-envelope $\Cmin(A(K))\to A(K)^{\ast\ast}$, still as a $\ast$-homomorphism. Because $A(K)^{\ast\ast}$ is a von Neumann algebra, this extends to a normal $\ast$-homomorphism 
\[
\pi:\Cmin(A(K))^{\ast\ast}\to A(K)^{\ast\ast}
\]
preserving $A(K)$.

Conversely, if $C=C^\ast(A(K))\subseteq A(K)^{\ast\ast}$ is the C*-subalgebra of $A(K)^{\ast\ast}$ generated by $A(K)$, then by the universal property of the C*-envelope, there is a $\ast$-homomorphism $C\to \Cmin(A(K))$ preserving $A(K)$. Upon identifying $C^{\ast\ast}\hookrightarrow A(K)^{\ast\ast}$, we have $C^{\ast\ast}=A(K)^{\ast\ast}$, because $A(K)^{\ast\ast}$ is generated as a von Neumann algebra by $A(K)$. So, double-dualizing the homomorphism $C\to \Cmin(A(K))$ gives a normal $\ast$-homomorphism
\[
\sigma:A(K)^{\ast\ast}\to \Cmin(A(K))^{\ast\ast}
\]
that preserves $A(K)$. Since $\pi$ and $\sigma$ are normal $\ast$-homomorphisms, and the copies of $A(K)$ generated $A(K)^{\ast\ast}$ and $\Cmin(A(K))^{\ast\ast}$ as von Neumann algebras, it follows that $\pi$ and $\sigma$ are mutual inverses and so $A(K)^{\ast\ast}\cong \Cmin(A(K))^{\ast\ast}$ naturally.
\end{proof}

An \textbf{nc Bauer simplex} has the additional property that the nc extreme points $\partial K$ are a closed set in the topology induced from the spectrum of $C(K)$, and $K$ is an nc Bauer simplex if and only if $A(K)$ is itself completely order isomorphic to a C*-algebra \cite[Theorem 10.5]{kennedy2022noncommutative}. The second author, Kim, and the third author obtained a noncommutative extension of this result. In \cite[Theorem 10.9]{kennedy2023nonunital}, they proved that the nonunital system $A(K,z)$ is completely order and norm isomorphic to a C*-algebra if and only if $K$ is an nc Bauer simplex and $z\in K_1$ is an nc extreme point of $K$. The corresponding characterization for nc (possibly non-Bauer) simplices is as follows.

\begin{theorem}\label{thm:nonunital_nc_simplex}
Let $(K,z)$ be a pointed compact nc convex set. The operator system $A(K,z)^{\ast\ast}$ is completely isometrically order isomorphic to a C*-algebra if and only if $K$ is an nc simplex and $z\in \partial K$.
\end{theorem}

\begin{proof}
The embedding of $A(K,z)$ into its partial unitization $A(K,z)^\sharp=A(K)$ double-dualizes to a completely isometric order injection of $A(K,z)^{\ast\ast}$ into $A(K)^{\ast\ast}$ of codimension one. And, $A(K)^{\ast\ast}$ is naturally viewed as a weak-$\ast$ closed subspace of the C*-algebra $B(K)\cong C(K)^{\ast\ast}$ of bounded nc functions on $K$. With this identification, we have $A(K)^{\ast\ast}=A(K,z)^{\ast\ast}+\C 1_{A(K)}$, from which it follows that $A(K)^{\ast\ast}$ coincides naturally with the partial unitization of $A(K,z)^{\ast\ast}$.

Let $K^{\ast\ast}$ denote the nc state space of $A(K)^{\ast\ast}$. Then, $K$ embeds via an affine nc homeomorphism onto a subset of $K^{\ast\ast}$, and we will denote the embedding $K\hookrightarrow K^{\ast\ast}$ by $x\mapsto x^{\ast\ast}$. Again, through the identification $A(K,z)^{\ast\ast}\subseteq A(K)^{\ast\ast}\subseteq B(K)$, it is straightforward to see that
\[
A(K,z)^{\ast\ast}=
A(K^{\ast\ast},z^{\ast\ast}).
\]
Indeed, the inclusion $A(K,z)^{\ast\ast}\subseteq A(K^{\ast\ast},z^{\ast\ast})$ is immediate. Conversely, an affine nc function $a\in A(K)^{\ast\ast}$ which vanishes on $z^{\ast\ast}$ can be approximated weak-$\ast$ by functions in $A(K)$, and--by adding a multiple of $1_{A(K)}=1_{A(K)^{\ast\ast}}$, by functions which vanish at $z$.

Because $K$ is an nc simplex, $A(K)^{\ast\ast}$ is isomorphic to a C*-algebra, and so its nc state space $K^{\ast\ast}$ is an nc Bauer simplex. So, it suffices to prove that $z^{\ast\ast}$ is an nc extreme point in $K^{\ast\ast}$. In the nc Bauer simplex $K^{\ast\ast}$ the nc extreme points are exactly irreducible representations of the C*-algebra $A(K)^{\ast\ast}$.

Since $z\in \partial K$, the point $z$ is a boundary representation of $A(K)$, which extends uniquely to an irreducible representation of $\Cmin(A(K))$. This extends further to a normal irreducible representation of $\Cmin(A(K))^{\ast\ast}$. By Lemma \ref{lem:double_dual_envelope}, we have $\Cmin(A(K))^{\ast\ast}\cong A(K)^{\ast\ast}$ naturally, and $z^{\ast\ast}$ is the unique normal extension of $z$. Therefore $z^{\ast\ast}$ is an irreducible representation and so is nc extreme in $A(K)^{\ast\ast}$. 

So, $K^{\ast\ast}$ is an nc Bauer simplex with nc extreme point $z^{\ast\ast}$, so \cite[Theorem 10.9]{kennedy2023nonunital} implies that
\[
A(K,z)^{\ast\ast}=A(K^{\ast\ast},z^{\ast\ast})
\]
is isomorphic to a C*-algebra.
\end{proof}

\begin{remark}
Note that if $K$ is an nc simplex, but $z$ is not nc extreme, then $A(K,z)$ 
\end{remark}

Since all C*-algebras are dualizable, we can conclude that all (possibly nonunital) C*-systems are dualizable.

\begin{corollary}
If $K$ is an nc simplex, and $z\in K_1\cap (\partial K)$ is an nc extreme point, then the operator system $A(K,z)$ is dualizable.
\end{corollary}

\begin{proof}
By Theorem \ref{thm:nonunital_nc_simplex}, the double dual $A(K,z)^{\ast\ast}$ is a C*-algebra, and therefore a dualizable operator system. Therefore, the triple dual $A(K,z)^{\ast\ast\ast}$ is an operator system. Since the natural map
\[
A(K,z)^{\ast}\to A(K,z)^{\ast\ast\ast}
\]
is a completely isometric order injection, the dual $A(K,z)^\ast$ embeds into an operator system and is therefore itself an operator system.
\end{proof}

Therefore, upon translating $(K,z)$ to $(K-z,0)$, if $K$ is an nc simplex containing $0$ as an nc extreme point, the nc geometric conditions in Corollary \ref{cor:dualizable_extrinsic} and Theorem \ref{thm:dualizable_intrinsic} hold for $K$.

\section{Dualizability for function systems}\label{sec:commutative}

By a \textbf{function system}, we mean a selfadjoint subspace of a commutative C*-algebra $C(X)$, for some compact Hausdorff space $X$. Classical Kadison duality \cite{kadison1951representation} asserts that function systems are categorically dual to (ordinary) compact convex sets. What follows is a commutative version of Theorem \ref{thm:dualizable_intrinsic}.(1)-(3). These results are known already as folklore, but we include proofs for completeness and to contrast the situation in Section \ref{sec:positive_generation}. C.K. Ng gives a more thorough discussion of what is known for function systems in \cite[Appendix A.2]{ng2022dual}.

\begin{proposition}\label{prop:dualizable_commutative}
Let $S$ be a (possibly nonunital) function system. The following are equivalent.
\begin{itemize}
\item [(1)] $S$ is positively generated, meaning $S=S^+-S^+$.
\item [(2)] $S$ is $\alpha$-generated for some $\alpha>0$.
\item [(3)] The dual $S^\ast$ is order and norm isomorphic to a function system.
\end{itemize}
Moreover, the isomorphism in (3) can be chosen to be a homeomorphism from the weak-$\ast$ topology to the topology of pointwise convergence on bounded sets.
\end{proposition}

\begin{proof}
The implication (2) $\implies$ (1) is immediate. If condition (3) holds, then $S^\ast$ is norm isomorphic to a function system, which is $1$-normal, and so $S^\ast$ is $\alpha$-normal for some $\alpha>0$. Therefore $S$ is $\alpha'$-generated for every $\alpha'>\alpha$. (See \cite[Theorem 2.1.5]{asimow2014convexity}, which is a classical result corresponding to part of Proposition \ref{prop:alpha_generated_dual}.)

\bigskip

If $S$ is positively generated, then it is a consequence of the Baire Category Theorem that $S$ is $\alpha$-generated, as in \cite[Theorem 2.1.2]{asimow2014convexity}. Sketching the proof, let
\[
B=
\conv(S^+_1\cup(-S^+_1)),
\]
where $S_1^+=B_1(S^+)$ is the unit ball of $S^+$. Then
\[
S^\sa\subseteq\bigcup_{\alpha\ge 1} n\overline{B},
\]
and so some $n\overline{B}$ has interior. By shifting and rescaling, we can arrange that
\[
S_1^\sa \subseteq
n\overline{B}
\]
for some $n>1$. Then, a series argument shows that $\overline{B}\subseteq (1+\epsilon)B$ for any $\epsilon>0$, and so $S_1^\sa\subseteq n(1+\epsilon)B$. That is, $S$ is $\alpha$-generated for any $\alpha>n$. Thus, (1) implies (2).

\bigskip

Now, suppose that $S$ is $\alpha$-generated, so that $S_1\subseteq \alpha\conv(S_1^+\cup(-S_1^+))$. Let $J:S\to S^{\ast\ast}$ be the natural embedding of $S$ into its double dual. Let $X\subseteq S^{\ast\ast}$ be the closure of (the image of) $S_1^+$ in the weak-$\ast$ topology of $S^{\ast\ast}$. Consider the linear map
\[
\rho:S^\ast \to C(X)
\]
which satisfies $\rho(f)(J(a))=f(a)$ for $a\in S_1^+$. By definition $\rho$ is an order isomorphism onto its range. Since $S_1^+\subseteq S_1$, the map $\rho$ is contractive. Given $f$ in $S^\ast$, and a selfadjoint $a\in S^\sa$, we can find $b,c \in S^+$ with $a=b-c$ and $\|b\|+\|c\|\le \alpha\|a\|$. Therefore, $b/\alpha\|a\|$ and $c/\alpha\|a\|$ are in $S_1^+$, and so
\begin{align*}
\|f(a)\| &\le
\|f(b)\|+\|f(c)\| \\ &=
\alpha \|a\|\left(\left\|f\left(\frac{b}{\alpha\|a\|}\right)\right\|+\left\|f\left(\frac{c}{\alpha\|a\|}\right)\right\|\right) \\ &=
\alpha \|a\|\left(\|\rho(f)\|+\|\rho(f)\|\right) =
2\alpha \|a\|\|\rho(f)\|.
\end{align*}
This proves that $\|\rho(f)\|\ge \|f\|/2\alpha$, and so $\rho$ is bounded below. Therefore $\rho$ is an order and norm isomorphism onto a function system. Since $X$ has the weak-$\ast$ topology, it also follows that $\rho$ is a weak-$\ast$ to pointwise homeomorphism on bounded sets.
\end{proof}

\begin{remark}\label{rem:dualizable_commutative_limitations}
Item (3) in Proposition \ref{prop:dualizable_commutative} cannot be extended to say \emph{completely} order and norm isomorphic, even if we replace (1) or (2) with the stronger hypothesis that $S$ is \emph{completely} $\alpha$-generated and so dualizable. To see why, if $S\subseteq C(X)$, then $S$ has its minimal operator space structure $S=\min(S)$, in the sense of \cite[Proposition 3.3.1]{effros2022theory}. Therefore, as an operator space $S^\ast = \max(S^\ast)$. Using the map $\rho$ in the proof of Proposition \ref{prop:dualizable_commutative}, $\rho(S)$ is a function system, and so $\rho(S)=\min(\rho(S))$ as operator spaces. Therefore, if $\rho$ was completely bounded below, it would induce a complete norm isomorphism between $\max(S)$ and $\min(S)$. If $S$ is infinite dimensional, then the maximal and minimal operator space structures are not completely equivalent \cite[Theorem 14.3]{paulsen2002completely}, and so $\rho$ cannot be completely bounded below.

By \cite[Theorem 3.9]{paulsen2002completely}, the map $\rho$ is completely positive. But, by definition of $\rho$, the map $\rho$ is a complete order isomorphism if and only if every positive map $S\to M_n$ is completely positive, but this is not true even for finite dimensional operator systems $S$.
\end{remark}

So, even if $S$ is a function system that is a dualizable operator system, its dual $S^\ast$ is typically not \emph{completel}y order and norm isomorphic to a function system, and never can be when $S$ is infinite dimensional. We don't know whether positive generation of a function system $S$ is enough to guarantee completely bounded positive generation and so dualizability. We leave this as an open question.

\begin{question}\label{ques:function_systems_generation}
If $S\subseteq C(X)$ is a positively generated function system, is $S$ completely $\alpha$-generated for some $\alpha>0$?
\end{question}

If so, then $S$ is dualizable if and only if it is positively generated. In Proposition \ref{prop:positively_generated} below, we show that positive generation actually guarantees positive generation at all matrix levels. If Question \ref{ques:function_systems_generation} has a negative answer, then by Proposition \ref{prop:positively_generated} below there is a function system $S$ for which each $M_n(S)$ is $\alpha_n$-generated, but the sequence $(\alpha_n)$ cannot be chosen to be bounded.

In Example \ref{eg:min_l1} below, we give a matrix ordered operator space which is positively generated , but not completely $\alpha$-generated for any $\alpha>0$. We do not know a function system with this property.

\section{Positive generation} \label{sec:positive_generation}

In Proposition \ref{prop:dualizable_commutative} above, we showed that for function systems, positive generation and bounded positive generation coincide. In this section, we discuss the noncommutative situation. First, we show that an operator system $S$ has \textbf{complete positive generation}, meaning $M_n(S)^\sa=M_n(S)^+-M_n(S)^+$ for all $n\ge 1$, if and only if $S$ is positively generated at the first level. In contrast to the classical situation, complete positive generation need not imply complete $\alpha$-generation. In Example \ref{eg:min_l1}, we give an example of a matrix ordered operator space which is positively generated but not completely $\alpha$-generated for any $\alpha>0$.

One might also consider the following weaker property. Call an ordered Banach space $E$ \textbf{approximately positively generated} if $E^+-E^+$ is dense in $E$. Note that even though the positive cone $E^+$ is closed, it need not be the case that $E^+-E^+$ is closed, even when $E$ is an operator space, as the following example shows.

\begin{example}\label{eg:approximately_positively_generated}
Let $S=C([0,1])$, and define $S^+$ to be the closed cone of functions which are both positive and convex. Then $S^+-S^+$ is dense in $S=C([0,1])$, because it contains all $C^2$ functions, but $S^+-S^+\ne S$, because the convex functions in $S^+$ are automatically differentiable almost everywhere on the interior $(0,1)$. So, $S$ is an ordered Banach space which is approximately positively generated, but not positively generated. In fact, $S$ is an operator system. Indeed, if we let
\[
K=\{\phi\in S^\ast\mid \|\phi\|\le 1\text{ and }\phi(S^+)\subseteq [0,\infty)\}
\]
be the classical quasistate space of $K$, then since every probability measure on $[0,1]$ lies in $K$, the natural map
\[
S\to A(K)
\]
into the continuous affine functions on $K$ is isometric and order isomorphic. That is, $S$ is isometrically order isomorphic to a nonunital function system, and so inherits an operator system structure.
\end{example}

There are many examples of the same kind as Example \ref{eg:approximately_positively_generated}. It suffices to take any function system $S$, and equip it with a new closed positive cone $P\subseteq S^+$ for which $P-P$ is not closed. In a private correspondence, Ken Davidson suggested another example in which $S=\C\oplus c_0$ is equipped with the new positive cone
\[
P=
\Big\{(t,(x_n)_{n\ge 1})\in \C\oplus c_0\mid t\ge 0, (x_n)_{n\ge 1}\ge 0, \text{ and }\sum_{n=1}^\infty x_n \le t\Big\}.
\]
Here, again $P-P$ is dense and not closed in $S$.

\begin{proposition}\label{prop:separates points}
Let $S$ be an operator system with quasistate space $K \subseteq S^\ast$. Then $S$ is approximately positively generated if and only if $S^+$ separates points in $K$.
\end{proposition}

\begin{proof}
If $S$ is densely spanned by its positives, then the positives must separate points in $K$. Conversely, suppose that $S$ is not positively generated. Then there exists an element $x \in S^\sa \setminus \overline{(S^+ - S^+)}$. By the Hahn-Banach Separation Theorem, there is a self-adjoint linear functional $\phi \in S^\ast$ so that for all $y \in S^+ - S^+$ we have 
\begin{align*}
	\phi(x) < \phi(y).
\end{align*}
But since $S^+-S^+$ is a real vector space, this implies that $\phi$ is identically zero on $S^+ - S^+$. Moreover, by the Hahn-Jordan decomposition theorem there are positive functionals $\phi^+, \phi^- \in E^d$ with $\phi = \phi^+-\phi^-$. Since $\phi(x) < 0$, the functionals $\phi^+$ and $\phi^-$ are necessarily distinct, but they are equal on $S^+-S^+$ and hence on $S^+$. Normalizing $\phi^\pm$ to obtain quasistates shows that $S^+$ does not separate quasistates.
\end{proof}

\begin{remark}
The Hahn-Jordan decomposition theorem ensures that, as an ordered vector space, the dual space $S^\ast$ is always positively generated.
\end{remark}

By the following result, if $S$ is positively generated then so are each of its matrix levels $M_n(S)$. Using Kadison Duality, each level $M_n(S)$ can itself be viewed as a function system by forgetting the rest of the matrix order, and so Proposition \ref{prop:dualizable_commutative} implies that each $M_n(S)$ is $\alpha_n$-generated for some $\alpha_n>0$. However, in order for $S$ to be dualizable, we would need the sequence $(\alpha_n)$ to be bounded.

\begin{proposition}\label{prop:positively_generated}
If $S$ is positively generated, then so is $M_n(S)$ for each $n$. That is, a positively generated operator system is automatically completely positively generated.
\end{proposition}

Before proving this, we will need a technical lemma which proves a much stronger statement in the finite dimensional setting.

\begin{lemma}
If $S$ is a finite dimensional and positively generated operator system, then it contains a matrix order unit.
\end{lemma}

\begin{proof}
Since $S$ is positively generated, then it admits a basis $B= \{ p_1,\ldots , p_m \}$ consisting of positive elements. We claim that $e:=\sum_{i=1}^m p_i$ is an order unit. For any $x$ in $S^\sa$, we can write $x$ uniquely as a real linear combination
\[
x = \sum_{i=1}^m \alpha_ip_i,
\]
and we define $\lambda_x:=\max\{1,|\alpha_1|,\ldots,|\alpha_m|\}$. It is clear that $\lambda_x e \pm x$ are positive in $S$, so $e$ is an order unit.

Next we let $n \ge 0$ and show that $e_n := e \otimes I_n$ is an order unit for $M_n(S)$, so fix an $X=(x_{ij})_{i,j=1}^n \in M_n(S)^\sa$. Since $E$ is positively generated, for every $i \le j$ we can decompose the corresponding entries of $X$ as
\begin{align*}
	x_{ij} = \real x_{ij}^+ - \real x_{ij}^- + i(\im x_{ij}^+ - \im x_{ij}^-).
\end{align*}
To find a large enough coefficient of $e_n$ to dominate $X$, we let 
\begin{align*}
\lambda_X := \lambda_d + \lambda_\real + \lambda_\im.
\end{align*}
Where $\lambda_d := \max\{\lambda_{x_{ii}}\}_{i=1}^n$, $\lambda_\real := \sum_{i<j}\lambda_{\real x_{ij}^+ + \real x_{ij}^-}$, and $\lambda_\im := \sum_{i<j}\lambda_{\im x_{ij}^+ + \im x_{ij}^-}$. Note that it makes sense to write $x_{ii}^\pm$ since the $x_{ii}$ must all be self-adjoint, as they lie on the diagonal of $X=X^*$.

Fix a concrete representation $S \hookrightarrow B(H)$ of $S$ as a norm closed and $\ast$-closed subspace of the bounded operators on a Hilbert space. We'll show that $\lambda_X e_n + X \ge 0$ concretely using inner products. Take an arbitrary vector $a = (a_i)_{i=1}^n \in H^n = \bigoplus_{i=1}^n H$, and compute
\begin{align*}
	\langle (\lambda_X e_n + X) a, a  \rangle
	= & \lambda_X \langle e_n a , a \rangle + \langle X a ,a\rangle \\
	= & \lambda_X \sum_{i=1}^n \langle ea_i,a_i \rangle
	+ \sum_{i=1}^n \langle x_{ii} a_i , a_i \rangle
	+ \sum_{i<j} \langle  x_{ij} a_j , a_i \rangle
	+ \langle x_{ji} a_i , a_j \rangle \\
	= & \lambda_X \sum_{i=1}^n \langle ea_i,a_i \rangle
	+ \sum_{i=1}^n \langle x_{ii} a_i , a_i \rangle
	+ \sum_{i<j} 2\real \langle  x_{ij} a_j , a_i \rangle \\
	= & \left( \lambda_d \sum_{i=1}^n \langle ea_i,a_i \rangle
	+ \sum_{i=1}^n \langle x_{ii} a_i , a_i \rangle \right) \\
	+ & \left( (\lambda_\real + \lambda_\im)\sum_{i=1}^n \langle ea_i,a_i \rangle +  \sum_{i<j} 2\real \langle  x_{ij} a_j , a_i \rangle \right) \\
	= & \left( \lambda_d \sum_{i=1}^n \langle ea_i,a_i \rangle
	+ \sum_{i=1}^n \langle x_{ii} a_i , a_i \rangle \right) \\
	+ & \left( \lambda_\real\sum_{i=1}^n \langle ea_i,a_i \rangle +  2\sum_{i<j} \real \langle  \real x_{ij} a_j , a_i \rangle \right) \\
	+ & \left( \lambda_\im\sum_{i=1}^n \langle ea_i,a_i \rangle -  2\sum_{i<j} \im \langle  \im x_{ij} a_j , a_i \rangle \right).
\end{align*}
For the remainder of the proof, we will show that each of the three terms above is non-negative. Starting with the first term,
\begin{align*}
	\lambda_d \sum_{i=1}^n \langle ea_i,a_i \rangle
	+ \sum_{i=1}^n \langle x_{ii} a_i , a_i \rangle 
	= & \sum_{i=1}^n \langle (\lambda_d e + x_{ii}) a_i, a_i \rangle \\
	\geq & \sum_{i=1}^n \langle (\lambda_{x_{ii}} e + x_{ii}) a_i, a_i \rangle \\
	\ge & \ 0,
\end{align*}
where the last inequality follows from the first paragraph of the proof.

To prove that the second term is non-negative, note
\begin{align*}
	& \lambda_\real\sum_{k=1}^n \langle ea_k,a_k \rangle +  2\sum_{i<j} \real \langle  \real x_{ij} a_j , a_i \rangle \\
	&\quad = \sum_{i<j}(\lambda_{\real x_{ij}^+ + \real x_{ij}^-}) \sum_{k=1}^n \langle ea_k,a_k \rangle +  2\sum_{i<j} \real \langle  \real x_{ij} a_j , a_i \rangle \\\
	&\quad = \sum_{i<j} (\lambda_{\real x_{ij}^+ + \real x_{ij}^-}) \sum_{k=1}^n \langle ea_k,a_k \rangle +  2\real \langle  \real x_{ij} a_j , a_i \rangle.
\end{align*}
We now show that for each pair $i<j$, the corresponding summand is non-negative:
\begin{align*}
	& (\lambda_{\real x_{ij}^+ + \real x_{ij}^-}) \sum_{k=1}^n \langle ea_k,a_k \rangle +  2\real \langle  \real x_{ij} a_j , a_i \rangle \\
	&\quad = (\lambda_{\real x_{ij}^+ + \real x_{ij}^-}) \sum_{k=1}^n \langle ea_k,a_k \rangle +  2\real \langle  (\real x_{ij}^+ - \real x_{ij}^-) a_j , a_i \rangle \\
	&\quad \ge (\lambda_{\real x_{ij}^+ + \real x_{ij}^-}) \langle ea_i,a_i \rangle
	+ (\lambda_{\real x_{ij}^+ + \real x_{ij}^-}) \langle ea_j,a_j \rangle 
	+  2\real \langle  (\real x_{ij}^+ - \real x_{ij}^-) a_j , a_i \rangle \\
	&\quad \ge \left( \langle \real x_{ij}^+ a_i, a_i \rangle
	+ \langle \real x_{ij}^+ a_j, a_j \rangle
	+ 2\real \langle \real x_{ij}^+ a_j, a_i \rangle \right) \\
	&\quad \quad + \left( \langle \real x_{ij}^- a_i, a_i \rangle
	+ \langle \real x_{ij}^- a_j, a_j \rangle
	- 2\real \langle \real x_{ij}^- a_j, a_i \rangle \right) \\
	&\quad = \langle \real x_{ij}^+ (a_i + a_j), a_i + a_j \rangle
	 + \langle \real x_{ij}^- (a_i - a_j), a_i - a_j \rangle \\
	&\quad \ge 0.
\end{align*}
The last inequality follows since each $\real x_{ij}^\pm$ is a positive operator. The proof that the third term is non-negative is similar.
\end{proof}

We now prove Proposition \ref{prop:positively_generated}

\begin{proof}[Proof of Proposition~\ref{prop:positively_generated}]

To show $M_n(S)$ is positively generated, fix $X=(x_{ij})_{i,j=1}^n \in M_n(S)^\sa$. Since $S$ is positively generated, each $x_{ij}$ can be written as a linear combination of four positives $\real x_{ij}^+$, $\real x_{ij}^-$, $\im x_{ij}^+$, and $\im x_{ij}^-$. Let $S_X$ denote the linear span of these positives, as $i$ and $j$ range from $1$ to $n$. Since $S_X$ is a finite dimensional operator system, by the previous lemma there is a matrix order unit $e_X \in S_X$ and in particular there is a constant $\lambda > 0$ so that both $\lambda 1_n \otimes e_X \pm X \ge 0 $. Since $X = (\lambda 1_n \otimes e_X + X)/2 - (\lambda 1_n \otimes e_X - X)/2$ and all entries are ultimately in $S$, this shows $M_n(S)$ is positively generated.
\end{proof}

So, complete positive generation coincides with positive generation at the first level. However, the following example shows that for matrix ordered operator spaces, positive generation at all matrix levels does not imply complete $\alpha$-generation for any $\alpha$. We do not know if this example is an operator system.

\begin{example}\label{eg:min_l1}
Any Banach space $E$ has a unique maximal and minimal system of $L^\infty$-matrix norms which give $E$ an operator space structure and restrict to the norm on $E$ at the first matrix level. We denote the resultant operator spaces by $\max(E)$ and $\min(E)$, respectively. There are natural operator space dualities $\max(E)^\ast=\min(E^\ast)$ and $\min(E)^\ast=\max(E)^\ast$ \cite[Section 3.3]{effros2022theory}.

We will consider the Banach space $\ell^1$ and its dual $\ell^\infty$. Because $\ell^\infty$ is a commutative $C^\ast$-algebra, we have $\ell^\infty=\min(\ell^\infty)$ \cite[Proposition 3.3.1]{effros2022theory}. The embedding $\ell^1\subseteq (\ell^\infty)^\ast$ gives a matrix ordered operator space structure on $\ell^1$, which coincides with the max norm $\ell^1=\max(\ell^1)$. Using the natural linear identifications
\[
M_n(\ell^\infty)=\ell^\infty(\N,M_n)\quad\text{and}\quad
M_n(\ell^1)=\ell^1(\N,M_n),
\]
the resultant positive cones in $\ell^\infty$ and $\ell^1$ consist of those sequences of matrices which are positive in each entry.

We will consider the minimal operator space $\min(\ell^1)$ equipped with the same matrix ordering as $\ell^1=\max(\ell^1)$. Because the matrix cones $M_n(\ell^1)^+=\ell^1(\N,M_n^+)$ are closed in the topology of pointwise weak-$\ast$ convergence, which is weaker than the topology induced by either the minimal or maximal norms on $M_n(\ell^1)$, the matrix cones $M_n(\ell^1)^+$ are closed in the minimal norm topology. Thus $\min(\ell^1)$ has the structure of a matrix ordered operator space. Because $M_n$ is $1$-generated, it follows that each $M_n(\min(\ell^1))=\ell^1(\N,M_n)$ is positively generated, so $\min(\ell^1)$ is completely positively generated.

However, we will show that $\min(\ell^1)$ is not completely $\alpha$-generated for any $\alpha>0$. We will do so using Proposition \ref{prop:alpha_generated_dual}, by proving the dual matrix ordered operator space $\min(\ell^1)^\ast = \max(\ell^\infty)$ (equipped with the usual matrix ordering on $\ell^\infty$) is not completely $\alpha$-normal for any $\alpha>0$. Since $\ell^\infty$ is infinite dimensional, the minimal and maximal matrix norms on $\ell^\infty$ are not completely equivalent \cite[Theorem 14.3]{paulsen2002completely}. Thus there is a sequence $x_k\in M_{n_k}(\ell^\infty)$ for which 
\[
\|x_k\|_\text{min} \le 1 \quad\text{and}\quad
\|x_k\|_\text{max}\ge k.
\]
In the C*-algebras $M_{n_k}(\ell^\infty)$, we can write each $x_k$ as a linear combination
\[
x_k =
(\Real x_k)^+ - (\Real x_k)^- + i (\Imag x_k)^+ -i(\Imag x_k)^-
\]
of positive elements $(\Real x_k)^\pm,(\Imag x_k)^\pm$ of min-norm at most $1$. Since $\|x_k\|_\text{max}>k$, by suitably choosing $y_k\in \{(\Real x_k)^\pm,(\Imag x_k)^\pm\}$, we can obtain a sequence of positive elements $y_k\in M_{n_k}(\ell^\infty)^+$ with
\[
\|y_k\|_\text{min} \le 1 \quad\text{and}\quad
\|y_k\|_\text{max}> k/4.
\]
Since the minimal norm on $M_{n_k}(\ell^\infty)$ is just the usual C*-algebra norm, we have $0\le y_k\le 1_{M_{n_k}(\ell^\infty)}$. Because the maximal norms satisfy the $L^\infty$-matrix norm identity, we have $\|1_{M_{n_k}(\ell^\infty)}\|_\text{max}=1$. Thus
\[
0\le y_k\le 1_{M_{n_k}(\ell^\infty)}, \quad \|1_{M_{n_k}(\ell^\infty)}\|_\text{max}\le 1,\quad \text{and} \quad
\|y_k\|_\text{max}> k/4
\]
for all $k\in\N$. So, $\ell^\infty$ is not completely $k/4$-normal, and taking $k\to \infty$ shows that $\ell^\infty$ cannot be completely $\alpha$-normal for any $\alpha>0$.
\end{example}

%\todo[inline]{This example gives a matrix ordered operator space which is positively generated but not boundedly positively generated. Is there an \emph{operator system} with this property? I don't know whether $\min(\ell^1)$ is an operator system, but it seems unlikely that it is.}

Example \ref{eg:min_l1} is a minimal example of this kind. One cannot restrict to the finite dimensional spaces $\ell^1_d$ and $\ell^\infty_d=(\ell^1_d)^\ast$ because the maximal and minimal norms on a finite dimensional Banach space are completely equivalent \cite[Theorem 14.3]{paulsen2002completely}, and so $\max(\ell^1_d)\cong \min(\ell^1_d)$ is a dualizable quasi-operator system.

\section{Permanence properties}\label{sec:permanence}

If $K=\coprod_{n\ge 1}K_n$ and $L=\coprod_{n\ge 1} L_n$ are compact nc convex sets, we denote by
\[
K\times L :=
\coprod_{n\ge 1} K_n\times L_n
\]
their levelwise cartesian product. In \cite{humeniuk2021jensen}, it was shown that $A(K\times L)$ is the categorical coproduct of the unital operator systems $A(K)$ and $A(L)$ in the category of unital operator systems with ucp maps as morphisms. The following result will let us assert a similar result in the pointed context, for nonunital operator systems.

\begin{proposition}\label{prop:product_affines}
Let $(K,z)$ and $(L,w)$ be pointed compact nc convex sets. Then $(K\times L,(z,w))$ is pointed, and there is a vector space isomorphism
\[
A(K\times L,(z,w))\cong
A(K,z)\oplus A(L,w).
\]
\end{proposition}

\begin{proof}
We will prove the result in the special case when $z=0$ and $w=0$ in the ambient spaces containing $K$ and $L$. The general case follows by translation. Define a linear map $A(K,z)\oplus A(L,w)\to A(K\times L,(z,w))$ by $(a,b)\mapsto a\oplus b$, where $(a\oplus b)(x,y):=a(x)+b(y)$ for $x\in K$, $y\in L$. Since $a(z)=0=b(w)$, it is easy to see that this map is injective. Given $c\in A(K\times L,(0,0))$, let $a(x)=c(x,0)$ and $b(y)=c(0,y)$ for $x\in K$, $y\in L$. Then since $c(0,0)=0$,
\begin{align*}
c(x,y) &=
2c\left(\frac{x}{2},\frac{y}{2}\right) \\ &=
2\left(\frac{c(x,0)}{2}+\frac{c(0,y)}{2}\right) \\ &=
a(x)+b(y)=(a\oplus b)(x,y).
\end{align*}
This proves that $A(K,0)\oplus A(L,0)\to A(K\times L,(0,0))$ is a linear isomorphism.

Now, it will follow from this isomorphism that $(K\times L,(z,w))$ is pointed. Let $\rho:A(K\times L,(z,w))\to M_n$ be any nc quasistate. Then
\[
\phi(a)=\rho(a\oplus 0)\quad\text{and}\quad\psi(b)=\rho(0\oplus b)
\]
define nc quasistates on $A(K,0)$ and $A(L,0)$, respectively. Because $(K,0)$ and $(L,0)$ are pointed, all nc quasistates are point evaluations, so we have $\phi(a)=a(x)$ and $\phi(b)=b(y)$ for some $(x,y)\in (K\times L)_n$ and all $a\in A(K,0)$, $b\in A(L,0)$. From linearity, it follows that $\rho$ is just point evaluation at $(x,y)$, so $(K\times L,(0,0))$ is pointed.
\end{proof}

\begin{definition}\label{def:coproduct}
Let $S$ and $T$ be operator systems with respective nc quasistate spaces $(K,0)$ and $(L,0)$. We define the \textbf{operator system coproduct} to be the vector space $S\oplus T$ equipped with the operator system structure such that
\[
S\oplus T\cong
A(K,0)\oplus A(L,0)\cong
A(K\times L,(0,0))
\]
is a completely isometric complete order isomorphism.
\end{definition}

Explicitly, the matrix norms on $S\oplus T$ satisfy
\[
\|(x,y)\|_{M_n(S\oplus T)}=
\sup\{\|\phi_n(x)+\psi_n(y)\|\mid \phi\in K,\psi\in L\}
\]
for $(x,y)\in M_n(S\oplus T)=M_n(S)\oplus M_n(T)$. The matrix cones just identify $M_n(S\oplus T)^+=M_n(S)^+\oplus M_n(T)^+$.

\begin{proposition}
The bifunctor $(S,T)\mapsto S\oplus T$ is the categorical coproduct in the category of operator systems with ccp maps as morphisms. That is, given any operator system $R$ and ccp maps $\phi:S\to R$ and $\psi:T\to R$, the linear map $\phi\oplus \psi:S\oplus T\to R$ is ccp.
\end{proposition}

\begin{proof}
This follows either by the explicit description of the matrix norms and order on $S\oplus T$, or by showing that $(K\times L,(0,0))$ is the categorical \emph{product} of $(K,0)$ and $(L,0)$ in the category of pointed compact nc convex sets, and using Theorem \ref{thm:nonunital_nc_kadison}.
\end{proof}

\begin{remark}
The operator space norm on $S\oplus T$ is neither the usual $\ell^\infty$-product nor the $\ell^1$-product of the operator spaces $S$ and $T$. For example, if
\[
K=L=\coprod_{n\ge 1}\{x\in M_n^+\mid 0\le x \le 1_n\}
\]
is the nc simplex generated by $[0,1]$, and $a\in A(K,0)$ is the coordinate function $a(x)=x$, then
\begin{align*}
\|a\oplus a\|_{A(K^2,(0,0))} &= 2 >\|a\oplus a\|_\infty\quad\text{and}\\
\|a\oplus (-a)\|_{A(K^2,(0,0))} &= 1<\|a\oplus a\|_1.
\end{align*}
\end{remark}

\begin{proposition}\label{prop:coproduct_dualizable}
Let $S$ and $T$ be operator systems. If $S$ and $T$ are dualizable, then $S\oplus T$ is dualizable.
\end{proposition}

\begin{proof}[Proof 1]
We will use Theorem \ref{thm:dualizable_intrinsic}. Let the nc quasistate spaces of $S$ and $T$ be $(K,0)$ and $(L,0)$, respectively. Then $(K-\R_+K)\cap \R_+K$ and $(L-\R_+L)\cap \R_+L$ are norm bounded. Checking that
\[
(K\times L-\R_+(K\times L))\cap \R_+(K\times L) \subseteq
((K-\R_+K)\cap \R_+K)\times ((L-\R_+L)\cap \R_+L)
\]
shows that $(K\times L-\R_+(K\times L))\cap \R_+(K\times L)$ is bounded, so $S\oplus T\cong A(K\times L,(0,0))$ is dualizable.
\end{proof}

It is also possible to give a proof of Proposition \ref{prop:coproduct_dualizable} using only Ng's bounded decomposition property, which appears in \ref{thm:dualizable_intrinsic}.(2).

More generally, we can form finite pushouts in the operator system category by taking pullbacks in the category of pointed compact nc convex sets.

\begin{definition}
Let
\[\begin{tikzcd}
	R \arrow[r,"\phi"] \arrow[d,"\psi"] & S \\
	T
\end{tikzcd}\]
be a diagram of operator systems with ccp maps as morphisms. Let $S$, $T$, and $R$, have respective quasistate spaces $(K,0)$, $(L,0)$, and $(M,0)$. We define the \textbf{pushout} $S\oplus_{R,\phi,\psi} T$ as the operator system
\[
A(K\times_{M,\phi^\ast,\psi^\ast} L,(0,0)),
\]
where
\[
K\times_{M,\phi^\ast,\psi^\ast}L =
\{(x,y)\in K\times L\mid \phi^\ast(x)=\psi^\ast(y)\}\subseteq K\times L,
\]
equipped with the natural maps
\begin{align*}
\iota_S:S\to S\oplus T &\to S\oplus_{R,\phi,\psi} T \quad\text{and}\\
\iota_T:T\to S\oplus T &\to S\oplus_{R,\phi,\psi} T
\end{align*}
which make the diagram
\begin{equation}\label{eq:pushout}\begin{tikzcd}
	R \arrow[r,"\phi"] \arrow[d,swap,"\psi"] & S \arrow[d,"\iota_S"]\\
	T \arrow[r,"\iota_T"] & S\oplus_{R,\phi,\psi} T
\end{tikzcd}\end{equation}
commute.
\end{definition}

When the morphisms $\phi$ and $\psi$ are understood, we will usually just write $S\oplus_R T$ and $K\times_M L$. Note that the coproduct $S\oplus T$ coincides with the pushout $S\oplus_0 T$ of the diagram
\[\begin{tikzcd}
	0 \arrow[r,"0"] \arrow[d,swap,"0"] & S \\
	T
\end{tikzcd}\]
as expected, where $0$ denotes the $0$ operator system.

To verify that $A(K\times_M L,(0,0))$ is an operator system, we need to show that:

\begin{proposition}
$(K\times_ML,(0,0))$ is pointed.
\end{proposition}
\begin{proof}
Let $\rho:A(K\times_M L,(0,0))\to M_n$ be an nc quasistate. Pulling $\rho$ back to $A(K\times L,(0,0))$ gives a point evaluation at some point $(x,y)\in K\times L$. It will suffice to show that $(x,y)\in K\times_M L$, in which case $\rho$ must be point evaluation at $(x,y)$.

We must show that $\phi^\ast(x)=\psi^\ast(y)$ in $M$. Given $a\in R\cong A(M,0)$. Since the diagram \eqref{eq:pushout} commutes, upon pulling back to $S\oplus T$, we have
\[
\rho(\iota_S\phi(a)) =
(\phi(a)\oplus 0)(x,y)=
(0\oplus \psi(a))(x,y) =
\rho(\iota_T\psi(a)),
\]
that is, $\phi(a)(x)=a(\phi^\ast(x))=\psi(a)(y)=a(\psi^\ast(y))$. Since $a\in R=A(M,0)$ was arbitrary, this proves $\phi^\ast(x)=\psi^\ast(y)$, so $(x,y)\in K\times_M L$.
\end{proof}

\begin{proposition}
The diagram \ref{eq:pushout} is a pushout in the category of operator systems with ccp maps as morphisms.
\end{proposition}

\begin{proof}
It is easiest to verify that the diagram
\[\begin{tikzcd}
	(K\times_M L,(0,0)) \arrow[r]\arrow[d] & (K,0) \arrow[d,"\phi^\ast"] \\
	(L,0) \arrow[r,"\psi^\ast"]& (M,0)
\end{tikzcd}\]
is a pullback in the category of pointed compact nc convex sets with pointed continuous affine nc functions as morphisms, where the unlabeled maps are just the coordinate projections. Checking this is fairly immediate, using the fact that the point-weak-$\ast$ topology on $K\times_M L\subseteq K\times L$ coincides with the restriction of the product topology. By the contravariant equivalence of categories Theorem \ref{thm:nonunital_nc_kadison}, it follows that \eqref{eq:pushout} is a pushout.
\end{proof}

\begin{proposition}
If $S$ and $T$ are dualizable operator systems, then any pushout $S\oplus_{R,\phi,\psi} T$ is also dualizable.
\end{proposition}

\begin{proof}
This follows from Proposition \ref{prop:coproduct_dualizable} combined with Corollary \ref{cor:subset_dualizable} used with the inclusion $(0,0)\subseteq K\times_M L\subseteq K\times L$.
\end{proof}

It follows by induction that any pushout of a finite family of dualizable operator systems is again dualizable.

\bibliography{references}
\bibliographystyle{amsplain}

\end{document}